\documentclass[draft,a4paper]{amsart}

\usepackage{pstricks,pst-node}
\usepackage{multirow}
\usepackage{amsmath,amsfonts,amssymb,amsthm}

\usepackage{url}

\title[Explicit isogenies of Jacobians of variable-separated curves]{
	Families of Explicitly Isogenous Jacobians
	\\
	of Variable-Separated Curves
}
\author{Benjamin Smith}
\address{
	INRIA Saclay--\^Ile-de-France 
	/ 
	Laboratoire d'Informatique de l'\'Ecole polytechnique (LIX),
	91128 Palaiseau Cedex, France
}
\email{smith@lix.polytechnique.fr}
\urladdr{http://www.lix.polytechnique.fr/Labo/Ben.Smith/}

\theoremstyle{definition}
\newtheorem{definition}{Definition}[section]
\newtheorem{example}{Example}[section]
\theoremstyle{remark}
\newtheorem{remark}{Remark}[section]
\theoremstyle{plain}
\newtheorem{theorem}{Theorem}[section]
\newtheorem{proposition}[theorem]{Proposition}
\newtheorem{lemma}[theorem]{Lemma}

\newtheorem{algorithm}[theorem]{Algorithm}

\newcommand{\WPS}[1]{\ensuremath{\mathbb{P}({#1})}}
\newcommand{\CC}{{\mathbb{C}}}

\newcommand{\QQ}{{\mathbb{Q}}}
\newcommand{\RR}{{\mathbb{R}}}
\newcommand{\QQbar}{{\overline{\mathbb{Q}}}}
\newcommand{\ZZ}{{\mathbb{Z}}}
\newcommand{\product}[3][{}]{\ensuremath{{#2}\!\times_{#1}\!{#3}}}
\newcommand{\XxY}{\product{X}{Y}}
\newcommand{\Jac}[1]{\ensuremath{{J}_{#1}}}

\newcommand{\Jacdual}[1]{\ensuremath{\widehat{J}_{#1}}}

\newcommand{\Newton}[1]{\ensuremath{\mathcal{N}({#1})}}
\newcommand{\Newtonint}[1]{\ensuremath{\mathcal{P}({#1})}}
\newcommand{\Bound}[1]{\ensuremath{b_{#1}}}

\newcommand{\family}[1]{\ensuremath{\mathcal{#1}}}

\newcommand{\Jacfamily}[1]{\ensuremath{{\family{J}}_{#1}}}

\newcommand{\Pic}[2][{}]{\ensuremath{\mathrm{Pic}^{#1}(#2)}}
\newcommand{\Div}[2][{}]{\ensuremath{\mathrm{Div}^{#1}(#2)}}

\newcommand{\Gal}[1]{{\mathrm{Gal}(#1)}}
\newcommand{\Hom}{{\mathrm{Hom}}}
\newcommand{\End}{{\mathrm{End}}}
\newcommand{\Mat}{{\mathrm{Mat}}}
\newcommand{\variety}[1]{\ensuremath{V\!\left({#1}\right)}}

\newcommand{\dualof}[1]{\ensuremath{{#1^{\dagger}}}}
\newcommand{\multiplication}[2][{}]{\ensuremath{[#2]_{#1}}}
\newcommand{\genus}[1]{\ensuremath{{g_{#1}}}}
\newcommand{\differentials}[1]{\ensuremath{\Omega({#1})}}
\newcommand{\closure}[1]{\ensuremath{\overline{#1}}}

\newcommand{\moduli}[1]{\ensuremath{\family{M}_{#1}}}

\newcommand{\integers}[1]{\ensuremath{\family{O}_{#1}}}
\newcommand{\cyclicgroup}[2][{}]{{(\ZZ/{#2}\ZZ)^{#1}}}
\newcommand{\isogenytype}[2]{\ensuremath{\cyclicgroup[#2]{#1}}}
\newcommand{\isogenytypetwo}[4]{\ensuremath{\cyclicgroup[#2]{#1}\!\times\!\cyclicgroup[#4]{#3}}}
\newcommand{\isogenytypethree}[6]{\ensuremath{\cyclicgroup[#2]{#1}\!\times\!\cyclicgroup[#4]{#3}\!\times\!\cyclicgroup[#6]{#5}}}

\newcommand{\genusfn}[2]{\ensuremath{g_{#2}({#1})}}

\newcommand{\rationalrep}[1]{\ensuremath{R(#1)}}
\newcommand{\differentialmatrix}[2][{}]{\ensuremath{D_{#1}({#2})}}
\newcommand{\differentialblock}[2]{\ensuremath{M_{#2}({#1})}}

\newcommand{\Tr}{\ensuremath{\mathrm{Tr}}}
\newcommand{\subgrp}[1]{\ensuremath{\left\langle{#1}\right\rangle}}

\newcommand{\ceilingof}[1]{\ensuremath{\left\lceil{#1}\right\rceil}}

\newcommand{\coker}{\ensuremath{\mathrm{coker}}}

\begin{document}

\begin{abstract}
	We construct six infinite series of families of pairs of curves
	$(X,Y)$
	of arbitrarily high genus, defined over number fields,
	together with an explicit isogeny $\Jac{X} \to \Jac{Y}$
	splitting multiplication by $2$, $3$, or $4$.
	The families are derived from Cassou--Nogu\`es and Couveignes'
	explicit classification
	of pairs $(f,g)$ of polynomials such that $f(x_1) - g(x_2)$
	is reducible.
\end{abstract}

\maketitle

\section{
	Introduction
}
\label{sec:introduction}

Our goal in this article is to give algebraic constructions
of explicit isogenies of Jacobians
of high-genus curves;
we are motivated by the scarcity of examples.
Isogenies of Jacobians 
are special
in genus $g > 3$,
in the sense that quotients of Jacobians 
are generally not Jacobians.
More precisely,
if $\phi: \Jac{X} \to \Jac{Y}$ 
is an isogeny of Jacobians
(that is, 
a geometrically surjective, finite homomorphism
respecting the canonical principal polarizations),
then
its kernel is a maximal $m$-Weil isotropic subgroup of 
the $m$-torsion $\Jac{X}[m]$
for some integer $m$.
On the other hand,
if
$S$ is a subgroup of $\Jac{X}[m]$
satisfying this same property,
then
the quotient $\Jac{X}\to\Jac{X}/S$
is an isogeny of principally polarized abelian varieties,
but in general $\Jac{X}/S$ is only isomorphic to a Jacobian
if the genus of $X$ is $\le 3$
(see~\cite{Oort--Ueno} and~\cite[Theorem 6]{Donagi--Livne}).

Nevertheless, 
families of non-isomorphic pairs of isogenous Jacobians
of high-genus curves exist: recently
Mestre~\cite{Mestre}
and the author~\cite{Smith-hyperelliptic}
have constructed families of hyperelliptic examples.
Here,
we extend the results of~\cite{Smith-hyperelliptic}
to derive new families of 
isogenies of non-hyperelliptic Jacobians
in arbitrarily high genus.
Theorem~\ref{theorem:main} summarises our results.

\begin{definition}
\label{definition:G-isogenies}
	If $\phi$ is an isogeny 
	with kernel isomorphic to a group $G$,
	then we say $\phi$ is a $G$-isogeny.\footnote{
	This is not the conventional notation,
	which replaces $G$ with a tuple of its abelian invariants;
	but it is much more useful in higher dimensions,
	where such tuples are typically very long.}
\end{definition}
\begin{definition}
\label{definition:genus-definition}
	For all positive integers $d$ and $n$, 
	we define the integer $\genusfn{d}{n}$ by
	\[
		\genusfn{d}{n} 
		:= 
		\frac{1}{2}\big((n-1)(d-1) - (\gcd(n,d) - 1)\big) 
		.
	\]
\end{definition}

\begin{theorem}
\label{theorem:main}
	For each integer $d > 1$
	and for each row of the following table,
	there exists 
	a $\nu$-dimensional family of explicit $G$-isogenies of Jacobians
	of curves of genus~$\genusfn{d}{n}$,
	defined over a CM-field of degree $e$;
	and
	if~$d$ is in~$S$,
	then the generic fibre 
	is 
	an isogeny
	of absolutely simple Jacobians
	(here $\mathcal{P}$ denotes the set of primes).
	\begin{center}
	\begin{tabular}{|r|l|c|l|l|}
		\hline
		$n$ & $\nu$ & $e$ 
			& $G$ & $S$ \\
		\hline
		\hline
		$7$ & $d$ & $2$ & 
			$\isogenytype{2}{\genusfn{d}{7}}$ & 
			$\ZZ_{\ge 2}$ \\ 
		\hline
		$11$ & $d-1$ & $2$ & 
			$\isogenytype{3}{\genusfn{d}{11}}$ &
			$\mathcal{P}\setminus\{11\}$ \\
		\hline
		$13$ & $d$ & $4$ & 
			$\isogenytype{3}{\genusfn{d}{13}}$ &
			$\ZZ_{\ge 2}$ \\
		\hline
		$15$ & $d$ & $2$ & 
			$\isogenytypetwo{4}{\genusfn{d}{15}-\genusfn{d}{5}-\genusfn{d}{3}}{2}{2\genusfn{d}{5}+2\genusfn{d}{3}}$ &
			$\mathcal{P}\setminus\{3,5,7\}$ \\
		\hline
		$21$ & $d-1$ & $2$ & 
			$\isogenytypetwo{4}{\genusfn{d}{21}-\genusfn{d}{3}}{2}{2\genusfn{d}{3}}$ &
			$\mathcal{P}\setminus\{3,5,7\}$ \\
		\hline
		$31$ & $d-1$ & $6$ & 
			$(\isogenytypethree{8}{}{4}{2}{2}{2})^{\genusfn{d}{31}/3}$ &
			$\mathcal{P}\setminus\{3,5,31\}$ \\
		\hline
	\end{tabular}
	\end{center}
\end{theorem}
\begin{proof}
	Follows from Propositions~\ref{proposition:phi-d-7}
	through~\ref{proposition:simplicity}.
\end{proof}

The proof of Theorem~\ref{theorem:main}
is organised as follows:
In \S\ref{section:construction},
we associate a family of pairs of curves $(\family{X},\family{Y})$
to each integer $d > 1$ and each pair of polynomials $(Q_X,Q_Y)$
such that $Q_X(x_1)-Q_Y(x_2)$ has a nontrivial factorization.
We also give a correspondence~$\family{C}$ 
on~$\product{\family{X}}{\family{Y}}$
inducing an explicit homomorphism 
$\phi_{\family{C}}: \Jacfamily{\family{X}}\to\Jacfamily{\family{Y}}$.
In \S\ref{sec:moduli} and \S\ref{section:representation}
we develop methods to determine
the number of moduli and
the kernel structure of $\phi_{\family{C}}$.
We recall the classification of 
Cassou--Nogu\`es and Couveignes~\cite{Cassou-Nogues--Couveignes}
in \S\ref{sec:polynomials},
and then apply our constructions to their polynomials
in~\S\S\ref{family:degree-7}-\ref{family:degree-31}.
Finally, 
in \S\ref{section:simplicity}
we list some values of $d$ where $\Jacfamily{\family{X}}$
and $\Jacfamily{\family{Y}}$ are known to be 
absolutely simple.

\subsection*{Connections to prior work}
The chief contribution of this work
is the construction of non-hyperelliptic families:
to our knowledge, 
all of the families of isogenies of Jacobians in genus $g > 3$ 
in the literature
are of hyperelliptic Jacobians.
The non-hyperelliptic families 
(those with~$d > 2$)
are all new.
Technically, the main improvement over~\cite{Smith-hyperelliptic}
is a more sophisticated approach to computing the action on differentials:
this allows us to treat all $d>1$ simultaneously,
and to determine the isogeny kernel structures
when $d > 2$ (the approach in~\cite{Smith-hyperelliptic}
uses an explicit description of the $2$-torsion specific to
hyperelliptic curves).
The hyperelliptic families (those with~$d = 2$)
have all appeared in earlier works:
The families $\phi_{2,7}$, $\phi_{2,11}$, $\phi_{2,13}$,
$\phi_{2,15}$, $\phi_{2,21}$, and $\phi_{2,31}$
are isomorphic to
the `linear construction' families in~\cite{Smith-hyperelliptic}.
The subfamily of~$\phi_{2,7}$ with $s_2 = 0$
and
the fibre of~$\phi_{2,11}$ at $s_2 = 0$
appear in Kux's thesis~\cite[Examples pp.59-60]{Kux--thesis}.
The endomorphisms of Proposition~\ref{proposition:RM-family}
with~$d = 2$
are isomorphic to those
described by Tautz, Top, and Verberkmoes~\cite{Tautz--Top--Verberkmoes}.

\subsection*{Notation}

Throughout, $K$ denotes a field of characteristic $0$
and $\zeta_n$ denotes a primitive $n^\mathrm{th}$
root of unity in $\QQbar \subset \closure{K}$.
Automorphisms of $K/\QQ$
act on polynomials over $K$
by acting on their coefficients.

\subsection*{Files}

Six files 
accompany this article
(\texttt{degree-n.m}, for \texttt{n} in $\{ 7, 11, 13, 15, 21, 31 \}$),
containing the coefficients of the polynomials and matrices
that appear in \S\S\ref{family:degree-7}-\ref{family:degree-31}.
(These objects are too big to be useful in printed form:
for example,
the matrix $\differentialblock{A_{31}}{30}$ in the proof of
Proposition~\ref{proposition:phi-d-31} is a~$\product{30}{30}$ matrix over a 
sextic number field, with $436$ nonzero entries.)
Each file is a program in the Magma language~\cite{Magma,MagmaHB},
but
they
should be easily adaptable
for use in other computational algebra systems;
in any case,
the reader need not be familiar with Magma to make use of the data.
If the files are not attached to this copy of the article,
then they may be found from the author's webpage.

\subsection*{Acknowledgements}

We thank 
John Voight,
for his suggestions at the 
\textsl{Explicit Methods in Number Theory} workshop
at the FWO in Oberwolfach, 2009;
the workshop organisers and the FWO itself, 
for the fruitful environment in which this work was begun;
Wouter Castryck, 
for his patient help and for pointing out Koelman's thesis;
and Frederik Vercauteren,
for sharing his implementation of the Gaudry--G\"urel 
point counting algorithm.

\section{
	Correspondences
}
\label{section:correspondences}

We begin with
a brief review of the theory of correspondences.
(See~\cite[\S11.5]{Birkenhake--Lange} 
and~\cite[\S16]{Fulton} 
for further detail.)

Let $X$ and $Y$ be (projective, irreducible, nonsingular) curves over a field $K$,
and let $C$ be a curve on the surface $\XxY$.
The natural projections 
from $\XxY$ 
restrict to morphisms
$\pi^{C}_{X}: C \to X$
and 
$\pi^{C}_{Y}: C \to Y$,
which in turn induce pullback and pushforward homomorphisms on divisor classes:
in particular, we have homomorphisms
\[
	(\pi^{C}_{X})^* : \Pic{X} \to \Pic{C} 
	\text{\quad and\quad }
	(\pi^{C}_{Y})_* : \Pic{C} \to \Pic{Y} .
\]
Both \((\pi^{C}_{X})^*\) and \((\pi^{C}_{Y})_*\) map
degree-$0$ classes to degree-$0$ classes,
and so induce homomorphisms of Jacobians
(and, \emph{a fortiori}, of principally polarized abelian varieties).
Composing, we get a homomorphism of Jacobians
\[
	\phi_C := (\pi^{C}_{Y})_*\circ(\pi^{C}_{X})^*
	: 
	\Jac{X} \longrightarrow \Jac{Y}
	;
\]
we say $C$ \emph{induces} $\phi_C$. 
We emphasize that
$\phi_C$ is completely explicit, given equations for~$C$:
we can evaluate $\phi_C(P)$ for any $P$ in $\Jac{X}$
by choosing a representative divisor 
from the corresponding class in $\Pic[0]{X}$,
pulling it back to $C$ with \((\pi^{C}_{X})^*\),
and pushing the result forward onto $Y$ with \((\pi^{C}_{Y})_*\).
Extending $\ZZ$-linearly so that $\phi_{C_1 + C_2} = \phi_{C_1} + \phi_{C_2}$,
we may take $C$ to be an arbitrary divisor on $\XxY$.
We call divisors on $\XxY$ \emph{correspondences}.

The map $C \mapsto \phi_C$;
defines
a homomorphism $\Div{\XxY} \to \Hom(\Jac{X},\Jac{Y})$;
its kernel
is generated by the principal divisors and the fibres of $\pi_X$ and $\pi_Y$.
The map is surjective:
every homomorphism $\phi: \Jac{X}\to\Jac{Y}$
is induced by some correspondence~$\Gamma_\phi$ on $\XxY$
(we may take
$\Gamma_\phi = (\phi\circ\alpha_{X}\times\alpha_{Y})^*\mu^*(\Theta_{Y})$,
where $\alpha_X: X \hookrightarrow \Jac{X}$ 
and $\alpha_Y: Y \hookrightarrow \Jac{Y}$
are the canonical inclusions,
$\Theta_Y$ is the theta divisor on $\Jac{Y}$,
and $\mu: \Jac{Y}\times\Jac{Y} \to \Jac{Y}$ is the subtraction map).
We therefore have an isomorphism
\[
	\Pic{\XxY} \cong \Pic{X}\oplus\Pic{Y}\oplus\Hom(\Jac{X},\Jac{Y}) .
\]

Exchanging the r\^oles of $X$ and $Y$ in the above,
we obtain the image of $\phi_C$ under the Rosati involution:
\[
	\dualof{\phi_C} = (\pi^{C}_{X})_*\circ(\pi^{C}_{Y})^* 
	:
	\Jac{Y} \longrightarrow \Jac{X}
	.
\]
(Recall that 
\( \dualof{\phi_C} := \lambda_{X}^{-1}\circ\hat\phi_C\circ\lambda_{Y} \),
where $\hat\phi_C:\Jacdual{Y}\to\Jacdual{X}$ is the dual homomorphism
and $\lambda_{X}:\Jac{X}\stackrel{\sim}{\to}\Jacdual{X}$
and $\lambda_{Y}:\Jac{Y}\stackrel{\sim}{\to}\Jacdual{Y}$
are the canonical principal polarizations.)

Composition of homomorphisms corresponds to
fibred products of correspondences:
if $X$, $Y$, and $Z$ are curves,
and $C$ and $D$
are correspondences on $\XxY$ and $\product{Y}{Z}$ respectively,
then 
$\product[Y]{C}{D}$ is a correspondence on $\product{X}{Z}$
and
\[
	\phi_D\circ\phi_C = \phi_{(\product[Y]{C}{D})} .
\]

Let $\differentials{X}$ and $\differentials{Y}$
denote the $\genusfn{d}{n}$-dimensional
$K$-vector spaces of regular differentials on $X$ and~$Y$, 
respectively.
The homomorphism $\phi_C: \Jac{X} \to \Jac{Y}$
induces a homomorphism of differentials
\[
	\differentialmatrix{\phi_C}: 
	\differentials{X} 
	\longrightarrow 
	\differentials{Y} 
\]
(see~\cite{Shimura} for details). 
The image of a regular differential $\omega$ on $X$
under $\differentialmatrix{\phi_{C}}$
is
\[
	\differentialmatrix{\phi_C}(\omega)
	=
	\Tr^{\differentials{C}}_{\differentials{Y}}(\omega) ,
\]
where the inclusion $\differentials{X} \hookrightarrow \differentials{C}$  
and the trace $\differentials{C}\to \differentials{Y}$
are induced by the natural inclusions
of $K(X)$ and $K(Y)$ in $K(C)$.
The map $\phi_C \mapsto \differentialmatrix{\phi_C}$
extends to
a faithful representation 
\[
	\differentialmatrix{\cdot}: 
	\Hom(\Jac{X},\Jac{Y}) 
	\to 
	\Hom(\differentials{X},\differentials{Y})
\]
(the faithfulness depends on the fact that $K$ has characteristic $0$).
We view differentials as row vectors,
and homomorphisms as matrices acting by multiplication on the right.
Composition of homomorphisms corresponds to matrix multiplication:
\[
	D(\phi_2\circ\phi_1) = D(\phi_1)D(\phi_2)
\]
for all $\phi_1:\Jac{X} \to \Jac{Y}$
and $\phi_2:\Jac{Y} \to \Jac{Z}$.
In particular,
if $Y = X$
then $D(\cdot)$ is a representation of rings;
in general,
$D(\cdot)$ is a representation 
of left $\End(\Jac{X})$- 
and right $\End(\Jac{Y})$-modules.

\begin{example}
\label{example:diagonal}
	Suppose that $X$
	is a curve with affine plane model $X: F(x,y) = 0$,
	and let $x_1,y_1$ and $x_2,y_2$
	denote the coordinate functions on the first and second
	factors of $\product{X}{X}$, respectively.
	Our first example of a nontrivial correspondence
	is the diagonal
	\[
		\Delta_X := \variety{y_1 - y_2, x_1 - x_2}
		\subset
		\product{X}{X}
		,
	\]
	which induces
	the identity map:
	$\phi_{\Delta_X} = \multiplication[\Jac{X}]{1}$.
	More generally,
	if $\psi$ is an automorphism of $X$,
	then $(\mathrm{Id}\times\psi)(X)$
	is a correspondence on $\product{X}{X}$ inducing~$\psi$.
\end{example}

\begin{example}
\label{example:hypersurface}
	Let $X$ and $Y$ be curves
	with affine plane models $X: F_X(x_1,y_1) = 0$
	and $Y: F_Y(x_2,y_2) = 0$.
	For any polynomial $A(x_1,y_1,x_2,y_2)$,
	the correspondence	
	$C = \variety{A}$ is rationally equivalent to a sum of fibres
	of $\pi_X$ and $\pi_Y$,
	and so induces the trivial homomorphism:
	on the level of degree-$0$ divisor classes,
	\[
		\phi_C\Big(\Big[{\sum_{P\in X(\closure{K})}n_P(P)}\Big]\Big)
		=
		\Big[\mathrm{div}\Big({\prod_{P\in X(\closure{K})}A(x_1(P),y_1(P),x_2,y_2)^{n_P}}\Big)\Big]
		=
		0
		.
	\]
	In particular,
	correspondences 
	inducing nonzero homomorphisms 
	must be
	cut out by more than one defining equation
	(cf.~Example~\ref{example:diagonal}).
\end{example}

\section{
	Variable-separated curves and correspondences
}
\label{section:construction}

Now let $X$ and $Y$ be variable-separated plane curves over $K$:
that is, we suppose that $X$ and $Y$ have affine plane models
\[
	X: P_X(y_1) = Q_X(x_1)
	\text{\quad and\quad }
	Y: P_Y(y_2) = Q_Y(x_2),
\]
where $P_X$, $Q_X$, $P_Y$, and $Q_Y$ are polynomials over $K$.
(This includes elliptic, hyperelliptic, and superelliptic $X$ and $Y$.)
We restrict our attention to the case 
where $P_X$, $P_Y$, $Q_X$ and $Q_Y$
are indecomposable:
that is, they cannot be written as compositions
of polynomials of degree at least two
(cf.~Remark~\ref{remark:indecomposability}).

Our aim is to give examples of correspondences
inducing nontrivial homomorphisms.
If 
$C = \variety{A}$ for some polynomial $A$,
then $\phi_{C} = 0$
(cf.~Example~\ref{example:hypersurface});
so we need to find divisors on $\XxY$
defined by at least two equations.
We investigate the simplest 
nontrivial case,
where each involves only two variables:
\[
	C = \variety{A(x_1,x_2),B(y_1,y_2)} \subset \XxY .
\]

We immediately reduce to the case where $P_X = P_Y$ 
and $B(y_1,y_2) = y_1-y_2$:
Let~$Z$ be the curve defined by
\(
	Z: P_X(v) = Q_Y(u) 
\),
and define correspondences 
\(
	C_1 = \variety{A(x_1,u), y_1 - v} 
\)
and
\(
	C_2 = \variety{u - x_2, B(v,y_2)}
\)
on $\product{X}{Z}$ and $\product{Z}{Y}$, respectively.
Then $C = \product[Z]{C_1}{C_2}$, so
\[
	\phi_C = \phi_{C_2}\circ\phi_{C_1} .
\]
Replacing $Y$ with $Z$
and $C$ with $C_1$
(or $X$ with $Z$ and $C$ with $C_2$), 
we reduce to the study of curves and correspondences defined by
\[
	X: P(y_1) = Q_X(x_1), 
	\quad 
	Y: P(y_2) = Q_Y(x_2),
	\quad 
	C = \variety{y_1-y_2,A(x_1,x_2)}
	.
\]

For $C$ to be one-dimensional,
we must have
\(A(x_1,x_2) | (Q_X(x_1) - Q_Y(x_2))\);
we will see in~\S\ref{sec:polynomials} 
that the existence of such a nontrivial factor
is special.
It is noted in~\cite[\S2.1]{Cassou-Nogues--Couveignes}
that if $Q_X$ and $Q_Y$ are indecomposable,
then the existence of a nontrivial~$A$
implies that 
$Q_X$ and $Q_Y$ have the same degree
\[
	n := \deg Q_X = \deg Q_Y ;
\]
and further that there exists some integer $r$ such that
\[
	r 
	= 
	\deg_{x_1}(A(x_1,x_2)) 
	= 
	\deg_{x_2}(A(x_1,x_2)) 
	= 
	\deg_{\mathrm{tot}}(A(x_1,x_2))
	,
\]
so we may write
\begin{equation}
\label{eq:A-form}
	A(x_1,x_2) 
	= 
	\sum_{i = 0}^r c_i(x_2)x_1^{r-i} 
	\quad \text{with } 
	\deg c_i \le i
	\text{ for all } 0 \le i \le r
	.
\end{equation}

We have no restrictions on $P$,
so we let it be (almost) generic\footnote{
	We could define $P_d$ to be the generic monic polynomial of degree $d$,
	but we can always change variables to remove its trace term 
	in characteristic zero,
	and this will be convenient in the sequel.
}:
for each integer $d > 1$ we let
$s_2,\ldots,s_d$ be free parameters,
and define $P_d$ to be the 
 polynomial
\[
	P_d(y) := y^d + s_2y^{d-2} + \cdots + s_{d-1}y + s_d .
\]
Note that $P_d$ is indecomposable.
Henceforward,
therefore, we consider families 
of curves $\family{X}$ and $\family{Y}$
and correspondences $\family{C}$
in the form
\begin{equation}
\label{eq:setup}
	\framebox{$
	\begin{array}{c}
	\family{X}: P_d(y_1) = Q_X(x_1),
	\quad \quad 
	\family{Y}: P_d(y_2) = Q_Y(x_2),
	\\
	{}
	\\
	\family{C} = \variety{y_1 - y_2, A(x_1,x_2)} 
	\subset 
	\product{\family{X}}{\family{Y}} , 
	\\
	{}
	\\
	\text{with } 
	Q_X \text{ and } Q_Y \text{ indecomposable of degree } n,
	\text{ and } A \text{ as in Eq. }(2) .
	\end{array}
	$}
\end{equation}
The families are parametrized by $s_2,\ldots,s_d$,
together with any parameters in the coefficients of $Q_X$ and $Q_Y$.
The special case $d = 2$,
which produces hyperelliptic families,
is the \emph{linear construction} of~\cite{Smith-hyperelliptic}
(with $s = -s_2$).

The Newton polygon of $\family{X}$ (and $\family{Y}$)
is
\[
	\Newton{d,n}
	= 
	\{ 
		(\lambda_1,\lambda_2) \in \RR_{\ge 0}^2
		: 
		d\lambda_1 + n\lambda_2 \le dn 
	\}
	.
\]
The families $\family{X}$ and $\family{Y}$
have (generically) nonsingular projective models
in the weighted projective plane $\WPS{d,n,1}$,
which is
the projective toric surface associated to $\Newton{d,n}$
(we see in~\cite{Reid} that $\WPS{d,n,1} = \WPS{d/m,n/m,1}$, 
where $m = \gcd(d,n)$).

We let 
$\Newtonint{d,n}$ denote the set
of integer interior points of the Newton polygon:
\[
	\Newtonint{d,n}
	=
	\{ (\lambda_1, \lambda_2) \in \ZZ_{> 0}^2 : 
		d \lambda_1 + n \lambda_2 < d n \}
	.
\]
The geometric genus of $\family{X}$ (and of $\family{Y}$) is equal to
$\#\Newtonint{d,n}$,
and it is easily verified that
if $g_n(d)$ is the function of Definition~\ref{definition:genus-definition},
then
\[
	\genus{\family{X}} 
	=
	\genus{\family{Y}}
	=
	\#\Newtonint{d,n}
	=
	\genusfn{d}{n} 
	.
\]

\begin{remark}
	Most known nontrivial examples of 
	explicit isogenies of Jacobians,
	including the isogenies of Richelot~\cite{Bost--Mestre}, 
	Mestre~\cite{Mestre}, and V\'elu~\cite{Velu}
	and the endomorphisms of Brumer~\cite{Brumer}
	and Hashimoto~\cite{Hashimoto},
	are \emph{not} induced by correspondences 
	in the form of Eq.~\eqref{eq:setup}.
	However,
	the explicit real multiplications 
	of Mestre~\cite{Mestre--RM} 
	and Tautz, Top, and Verberkmoes~\cite{Tautz--Top--Verberkmoes}
	are in the form of Eq.~\eqref{eq:setup}.
\end{remark}

\begin{remark}
\label{remark:rational-function-generalization}
	Our construction generalizes readily to the case where $P_d$,
	$Q_X$, and~$Q_Y$ are rational functions
	instead of polynomials.
	While this yields many more families,
	it also complicates the algorithmic aspects
	of our constructions below.
\end{remark}

\section{
	Isomorphisms and Moduli
}
\label{sec:moduli}

We want to compute the number of moduli of~$\family{X}$:
that is, the dimension of the image of~$\family{X}$
in the moduli space~$\moduli{\genusfn{d}{n}}$ 
of curves of genus~$\genusfn{d}{n}$ over $\closure{K}$.
By Torelli's theorem,
this is also the dimension of the image of the family $\phi_{\family{C}}$
in the appropriate moduli space of homomorphisms of principally polarized abelian varieties.

We will adapt the methods of Koelman's thesis~\cite{Koelman}
to compute the number of moduli.
Up to automorphism, 
we can determine the form of 
the polynomials defining any isomorphism between curves in $\family{X}$
by considering column structures and column vectors 
on the projective toric surface
associated to $\Newton{d,n}$,
where $\family{X}$ has a convenient nonsingular embedding
(see~\cite{Bruns--Gubeladze}, \cite{Castryck--Voight}, and~\cite{Koelman}
for details).

More specifically,
for $d > 2$,
we embed $\family{X}$ in $\WPS{d,n,1}$.
The $\closure{K}$-isomorphisms between distinct curves in $\family{X}$
must then take the form
\begin{equation}
\label{eq:iso-form}
	(x,y)
	\longmapsto
	(ax + b, ey)
\end{equation}
for some $a$, $b$, $e$ in $\closure{K}$ with $a$ and $e$ nonzero.
When $d = 2$,
it is more convenient to embed $\family{X}$ in $\WPS{1,\genusfn{d}{n}+1,1}$;
the $\closure{K}$-isomorphisms 
must then take the form
\begin{equation}
\label{eq:hyperelliptic-iso-form}
	(x,y) 
	\longmapsto 
	((ax + b)/(cx + d),ey/(cx + d)^{(\genusfn{d}{n}+1)})
\end{equation}
for $a$, $b$, $c$, $d$, and $e$ in $\closure{K}$ with
$e$ and $ad - bc$ nonzero.

\begin{lemma}
\label{lemma:family-dimension-no-t}
\label{lemma:family-dimension-with-t}
	Let $d > 1$ be an integer,
	$K$ a subfield of $\CC$,
	and 
	\( f(x) = \sum_{i = 0}^n f_ix^{n-i} \)
	a polynomial over $K$
	or $K(t)$,
	where $t$ is a free parameter,
	such that $\genusfn{d}{n}>1$ and
	\[
		(i)\  f_0 = 1,\quad 
		(ii)\  f_1 = 0,\quad 
		(iii)\  f_2 \not= 0,\quad 
		\text{and }\ 
		(iv)\  f_3 = \kappa f_2 \ \text{ for some } \kappa\in K
		.
	\]
	Let $\family{X}$ be the family defined by $\family{X}: P_d(y) = f(x)$.
	Then 
	\begin{enumerate}
	\item	if $f_i$ is in $K(t)\setminus K$
		for some $2 \le i < n$,
		then $\family{X}$
		has $d$ moduli;
	\item	otherwise,
		$\family{X}$
		has $d-1$ moduli.
	\end{enumerate}
\end{lemma}
\begin{proof}
	Let $\family{U}$ be the open subfamily of $\family{X}$
	where $s_2,\ldots,s_d$ are all nonzero.
	It suffices to show that the intersection of $\family{U}$
	with the isomorphism class of any curve in $\family{U}$
	is finite.
	First, observe that $\family{U}$ has no nontrivial constant subfamilies:
	the parameters $s_1,\ldots,s_d$ (and $t$ in Case (1)) appear
	in distinct coefficients of the defining equation of $\family{U}$.
	It is enough, therefore, 
	to show that there are only finitely many possible
	defining equations for isomorphisms from a fixed curve in $\family{U}$
	to other curves in~$\family{U}$.
	Every such isomorphism
	has the form of 
	Eq.~\eqref{eq:iso-form} 
	(or
	Eq.~\eqref{eq:hyperelliptic-iso-form}
	for $d = 2$).
	But the defining equation of the codomain curve
	must satisfy (i) through (iv),
	which determine $e$ and $ax + b$ (or $(ax + b)/(cx + d)$)
	up to a finite number of choices.
\end{proof}

\section{
	The representation on differentials
}
\label{section:representation}

Let $\family{X}$, $\family{Y}$, and $\family{C}$ 
be as in Eq.~\eqref{eq:setup}.
We want to make the representation $\differentialmatrix{\phi_{\family{C}}}$
of~\S\ref{section:correspondences} completely explicit,
with a view to determining the structure of $\ker\phi_C$.
It suffices to consider the generic fibres $X$, $Y$, and $C$
of $\family{X}$, $\family{Y}$, and $\family{C}$ respectively.
In this section,
$K$ denotes the field of definition of $X$, $Y$, and~$C$.

First, we partition $\Newtonint{d,n}$ into disjoint ``vertical'' slices:
\begin{equation}
\label{eq:Newton-slicing}
	\Newtonint{d,n}
	= 
	\bigsqcup_{i = 1}^{\Bound{d,n}}
	\{ 
		(i, j) : 
		1 \le j \le p_{d,n}(i)
	\}
	,
\end{equation}
where
\[
	\Bound{d,n} 
	:= 
	\max\{i: (i,j)\in\Newtonint{d,n}\}
	=
	\ceilingof{(1-1/n)d} - 1 
\]
and 
\[
	p_{d,n}(i) := \ceilingof{(1-i/d)n} - 1 
	\quad \text{ for } 1\le i \le \Bound{d,n}. 
\]
We fix a basis for the spaces of regular differentials on $X$ and $Y$:
\[
	\differentials{X}
	=
	\subgrp{ \omega_{i,j} : (i,j) \in \Newtonint{d,n} }
	\text{\quad and\quad }
	\differentials{Y} 
	= 
	\subgrp{ \omega_{i,j}' : (i,j) \in \Newtonint{d,n} }
	,
\]
where
\[
	\omega_{i,j} := \frac{y_1^{i-1}}{P_d'(y_1)}d(x_1^j)
	\text{\quad and\quad }
	\omega_{i,j}' := \frac{y_2^{i-1}}{P_d'(y_2)}d(x_2^j)
	.
\]
This fixes isomorphisms of $\differentials{X}$ and $\differentials{Y}$
with $K^{\genusfn{d}{n}}$;
we view regular differentials on $X$ and $Y$
as row $\genusfn{d}{n}$-vectors over $K$.
If we define subspaces
\[
	\differentials{X}_i
	:= 
	\subgrp{ 
		\omega_{i,j}
		:
		1 \le j \le p_{d,n}(i)
	}
	\text{\quad and\quad }
	\differentials{Y}_i
	:= 
	\subgrp{ 
		\omega_{i,j}'
		:
		1 \le j \le p_{d,n}(i)
	}
\]
for $1 \le i \le \Bound{d,n}$,
then the partition of Eq.~\eqref{eq:Newton-slicing}
induces direct sum decompositions
\begin{equation}
\label{eq:differential-decomposition}
	\differentials{X}
	=
	\bigoplus_{i=1}^{\Bound{d,n}} 
	\differentials{X}_{i}
	\text{\quad and\quad }
	\differentials{Y}
	=
	\bigoplus_{i=1}^{\Bound{d,n}} 
	\differentials{Y}_{i}
	.
\end{equation}

Since $y_1 = y_2$ in $K(C)$,
the image of $\omega_{i,j}$
under $\differentialmatrix{\phi_{C}}$
is
\[
	\differentialmatrix{\phi_C}(\omega_{i,j})
	=
	\Tr^{\differentials{C}}_{\differentials{Y}}\Big(\frac{y_1^{i-1}d(x_1^j)}{P_d'(y_1)}\Big)
	=
	\frac{y_2^{i-1} d\big(\Tr^{K(C)}_{K(Y)}(x_1^j)\big) }{P_d'(y_2)}
	=
	\frac{y_2^{i-1}}{P_d'(y_2)}dt_j ,
\]
where
\[
	t_j 
	:= 
	\Tr^{K(x_2)[x_1]/(A(x_1,x_2))}_{K(x_2)}(x_1^j)
	.
\]
By definition,
$t_j$ is the $j^\mathrm{th}$ power-sum symmetric polynomial
in the roots of $A$ viewed as a polynomial in $x_1$ over $\closure{K(x_2)}$;
but for $k > 0$,
the $k^\mathrm{th}$ elementary symmetric polynomial in these same roots
is equal to $(-1)^kc_k/c_0$,
where $c_k$ and $c_0$ are as in
Eq.~\eqref{eq:A-form}.
We can therefore compute the $t_j$
using the Newton--Girard recurrences
\[
	t_1 = -\frac{c_1}{c_0},
	\quad
	t_2 = -\frac{2c_2 + t_1c_1}{c_0},
	\quad
	\cdots,
	\quad
	t_j = -\frac{jc_j + \sum_{k=1}^{j-1}c_jt_{j-k}}{c_0}
	.
\]
Equation~\eqref{eq:A-form} implies
$\deg t_j \le \deg c_j \le j$, so
expanding $t_j$ in terms of $x_2$ we write
\[
	t_j = \sum_{k= 0}^j\mu_{j,k}x_2^{k} .
\]
In terms of differentials
we have 
\(
	dt_j 
	= 
	d(\sum_{k=0}^j\mu_{j,k}x_2^{k}) 
	= 
	\sum_{k=1}^j\mu_{j,k}d(x_2^k)
\),
so
\[
	\differentialmatrix{\phi_C}(\omega_{i,j})
	=
	\sum_{k=1}^j \mu_{j,k}\omega_{i,k}'
	.
\]
In particular, $\differentialmatrix{\phi_C}$ respects
the decomposition of Eq.~\eqref{eq:differential-decomposition}:
that is,
\begin{equation}
\label{eq:restriction-mapping}
	\differentialmatrix{\phi_C}(\differentials{X}_i) 
	\subset 
	\differentials{Y}_i
\end{equation}
for all \(1 \le i \le \Bound{d,n} \).
For each $0 < k < n$, we define a matrix
\[
	\differentialblock{A}{k}
	:= 
	\left( \begin{array}{rrrcl}
		\mu_{1,1} & 0 & 0 & \cdots & 0 \\
		\mu_{2,1} & \mu_{2,2} & 0 & \cdots & 0 \\
		\mu_{3,1} & \mu_{3,2} & \mu_{3,3} & \cdots & 0 \\
		\vdots    & \vdots    & \vdots &       & \vdots \\
		\mu_{k,1} & \mu_{k,2} & \mu_{k,3} & \cdots & \mu_{k,k} \\
		\end{array}
	\right)
\]
representing
\( 
	\differentialmatrix{\phi_{C}}|_{\differentials{X}_{k}}
	:
	{\differentials{X}_{k}}
	\to
	{\differentials{Y}_{k}}
\).
Combining Equations~\eqref{eq:differential-decomposition}
and~\eqref{eq:restriction-mapping},
we have
\begin{equation}
\label{eq:differential-matrix-decomposition}
	\differentialmatrix{\phi_C}
	=
	\bigoplus_{i=1}^{\Bound{d,n}}\differentialblock{A}{p_{d,n}(i)}
	.
\end{equation}
The $i^\mathrm{th}$ summand
in Eq.~\eqref{eq:differential-matrix-decomposition}
is (by definition) 
the upper-left $\product{p_{d,n}(i)}{p_{d,n}(i)}$ submatrix 
of $\differentialblock{A}{n-1}$,
because 
$p_{d,n}(i) \le n-1$ for all~$i$.
Hence, we need only compute $\differentialblock{A}{n-1}$
to determine $\differentialmatrix{\phi_C}$ for arbitrary $d$.

\begin{algorithm}
\label{algorithm:differential-matrix-block}
	Computes the maximal block $\differentialblock{A}{n-1}$
	of the matrix $\differentialmatrix{\phi_C}$.
\begin{description}
\item[Input]
	An integer $n \ge 2$,
	and 
	a polynomial $A(x_1,x_2)$ over $K$ in the form of Eq.~\eqref{eq:A-form}:
	that is, 
	\(
	A(x_1,x_2) 
	= 
	\sum_{i = 0}^r c_i(x_2)x_1^{r-i} 
	\)
	with
	\(\deg c_i \le i\)
	for all~$i$. 
\item[Output]
	The matrix $\differentialblock{A}{n-1}$.
\item[1]
	Let $c_i := 0$ for $r < i < n$.
\item[2]
	For i in $(1,\ldots,n-1)$ do
	\begin{description}
	\item[2a] 
		Set $t_i := -\big( i c_i + \sum_{j=1}^{i-1}c_jt_{i-j}\big)/c_0$ .
	\item[2b]
		For $j$ in $(1,\ldots,n-1)$,
		let $\mu_{i,j} \in K$ be 
		the coefficient of $x_2^{j}$ in $t_i$.
	\end{description}
\item[3]
	Return the matrix $(\mu_{i,j})$.
\end{description}
\end{algorithm}

The representation of the Rosati dual $\dualof{\phi_C}$ is 
\(
	\differentialmatrix{\dualof{\phi_C}}
	=
	\bigoplus_{i=1}^{\Bound{d,n}}\differentialblock{A(x_2,x_1)}{p_{d,n}(i)}
\).
We make the following definition
for notational convenience.

\begin{definition}
We define an involution $\tau$ on $K[x_1,x_2]$
by 
\[
	\tau(A(x_1,x_2)) := A(x_2,x_1) .
\]
\end{definition}

\begin{lemma}
\label{lemma:isogeny-classification}
	With the notation above,
	if 
	\(
		\differentialblock{A}{n-1}
		\differentialblock{\tau(A)}{n-1}
		=
		mI_{n-1}
	\)
	for some integer $m$,
	then
	\(
		\dualof{\phi_C}\phi_C = \multiplication[\Jac{X}]{m}
	\)
	(that is, $\phi_C$ splits multiplication-by-$m$ on~$\Jac{X}$).
	Further,
	if $m$ is squarefree, 
	then $\phi_C$ is a
	\( \isogenytype{m}{\genusfn{d}{n}} \)-isogeny.
\end{lemma}
\begin{proof}
	We have
	\(
		\differentialmatrix{\dualof{\phi_C}\phi_C} 
		=  
		\differentialmatrix{\phi_C}\differentialmatrix{\dualof{\phi_C}}
	\),
	so Eq.~\eqref{eq:differential-matrix-decomposition} 
	implies
	\[
		\differentialmatrix{\dualof{\phi_C}\phi_C}
		=
		\bigoplus_{i=1}^{\Bound{d,n}}
		\big(
		\differentialblock{A}{p_{d,n}(i)}
		\differentialblock{\tau(A)}{p_{d,n}(i)}
		\big)
		.
	\]
	As we noted above,
	$\differentialblock{A}{k}$
	is the upper-left $k\!\times\!k$ submatrix 
	of $\differentialblock{A}{n-1}$
	for all~$k$. 
	Both $\differentialblock{A}{n-1}$
	and $\differentialblock{\tau(A)}{n-1}$
	are lower-triangular,
	so
	$\differentialblock{A}{k}\differentialblock{\tau(A)}{k}$
	is the upper-left $k\!\times\!k$ submatrix of
	$\differentialblock{A}{n-1}\differentialblock{\tau(A)}{n-1}$,
	which is 
	$mI_{n-1}$ by hypothesis.
	Hence
	\(
		\differentialblock{A}{i}
		\differentialblock{\tau(A)}{i} 
		= 
		mI_i
	\) for all $1 \le i \le \Bound{d,n}$, 
	and therefore
	\[
		\differentialmatrix{\dualof{\phi_C}\phi_C}
		=
		\bigoplus_{i = 1}^{\Bound{d,n}} mI_i
		=
		mI_{\genusfn{d}{n}}
		.
	\]
	The faithfulness of $\differentialmatrix{\cdot}$ 
	implies $\dualof{\phi_C}\phi_C = \multiplication[\Jac{X}]{m}$,
	proving the first assertion.
	The kernel of $\phi_C$
	must be a maximal subgroup of $\Jac{X}[m]$
	with respect to the property of being 
	isotropic for the $m$-Weil pairing;
	when $m$ is squarefree,
	the second assertion follows from this
	together with the nondegeneracy
	of the Weil pairing.
\end{proof}

Lemma~\ref{lemma:isogeny-classification}
determines the kernel structure of isogenies
splitting multiplication by a squarefree integer.
In \S\S\ref{family:degree-15}-\ref{family:degree-31},
we will derive isogenies splitting multiplication by~$4$ and~$8$;
we will need another method to determine their kernel structures.
It is helpful to specialize to an isogeny defined over a number field,
and then to view the specialized isogeny 
as an isogeny of complex abelian varieties.

Suppose that $K$ is a number field.
Fix an embedding of $K$ into $\CC$,
and let $\sigma$ denote complex conjugation;
enlarging $K$ if necessary,
we assume $K^\sigma = K$.
Viewing~$\Jac{X}$ and $\Jac{Y}$ as complex tori,
there exist coordinates
on~$\CC^{\genusfn{d}{n}}$
and lattices $\Lambda_X$ and $\Lambda_Y$ in $\CC^{\genusfn{d}{n}}$
such that 
the analytic representation 
$S(\phi_{C}): \CC^{\genusfn{d}{n}}\to\CC^{\genusfn{d}{n}}$
and the rational representation $R(\phi_{C}): \Lambda_X \to \Lambda_Y$
are given by the matrices
\begin{equation}
\label{eq:representations}
	S(\phi_C) = \differentialmatrix{\phi_{C}}
	\quad \text{ and }\quad 
	R(\phi_C) = \left(\begin{array}{cc} 
			\differentialmatrix{\phi_C} & 0 \\ 
			0 & \differentialmatrix{\phi_C}^\sigma \\ 
	\end{array}\right)
	.
\end{equation}
We will compute the structure of $\ker(\phi_{C})$
using the relation
\begin{equation}
\label{eq:kerphiC}
	\ker(\phi_{C})
	\cong
	\coker(R(\phi_{C}))
	\cong
	\Lambda_Y/R(\phi_{C})(\Lambda_X)
	.
\end{equation}
The first step is a restriction of scalars from $K$ to $\QQ$,
since we do not know \emph{a priori} 
how elements of~$K$ should act 
on our unknown lattices $\Lambda_X$ and $\Lambda_Y$.
Suppose that~$R(\phi_C)$ is defined over the ring $\integers{K}$ of integers of $K$
(it is sufficient that $A$ be a polynomial over $\integers{K}$).
Fixing a $\ZZ$-basis $\gamma_1, \ldots, \gamma_e$ of~$\integers{K}$,
we have a faithful representation
$\rho: \integers{K} \to \Mat_{e\times e}(\ZZ)$
(made explicit in Algorithm~\ref{algorithm:Gak}),
which extends to a homomorphism
\[
	\rho_*: \Mat_{2\genusfn{d}{n}\times2\genusfn{d}{n}}(\integers{K}) 
	\longrightarrow 
	\Mat_{2e\genusfn{d}{n}\times2e\genusfn{d}{n}}(\ZZ)
\]
mapping a matrix $(a_{i,j})$ to the block matrix $(\rho(a_{i,j}))$.
We then have
\begin{equation}
\label{eq:rhostarquotient}
	(\Lambda_Y/R(\phi_C)(\Lambda_X))^e 
	\cong 
	\ZZ^{2e\genusfn{d}{n}}/\rho_*(R(\phi_C))(\ZZ^{2e\genusfn{d}{n}})
	,
\end{equation}
so we can compute the isomorphism type of $(\ker\phi_C)^e$
by computing the elementary divisors of $\rho_*(R(\phi_C))$.
Combining Equations~\eqref{eq:differential-matrix-decomposition}
and~\eqref{eq:representations},
and applying $\rho_*$,
we have
\begin{equation}
\label{eq:rhostarimage}
	\rho_*(\rationalrep{\phi_C})
	= 
	\bigoplus_{i=1}^{\Bound{d,n}} 
		\rho_*\Big(
			\differentialblock{A}{p_{d,n}(i)}
			\oplus
			\differentialblock{A}{p_{d,n}(i)}^\sigma
		\Big)
	.
\end{equation}
For each $1 \le k \le n-1$,
we define 
\[
	G(A,k) 
	:= 
	\ZZ^{2ek}
	/
	\big(
		\rho_*\big(\differentialblock{A}{k}
		\oplus
		\differentialblock{A}{k}^\sigma\big)(\ZZ^{2ek})
	\big)
	;
\]
then combining Equations~\eqref{eq:kerphiC}, \eqref{eq:rhostarquotient}, 
and~\eqref{eq:rhostarimage},
we have
\begin{equation}
\label{eq:kernel-relation}
	(\ker(\phi_C))^e
	\cong
	\bigoplus_{i=1}^{\Bound{d,n}}G(A,p_{d,n}(i))
	.
\end{equation}
We can use this relation to deduce the structure of $\ker(\phi_C)$.

\begin{algorithm}
\label{algorithm:Gak}
Computes the sequence $(G(A,k))_{k=1}^{n-1}$.
\begin{description}
\item[Input] 	A polynomial $A \in \integers{K}[x_1,x_2]$,
		and an integer $n$.
\item[Output] 	The sequence of groups $G(A,k)$ for $1 \le k \le n-1$.
\item[1]	Compute $\differentialblock{A}{n-1}$
		using Algorithm~\ref{algorithm:differential-matrix-block}.
\item[2] 	Set $e := [K : \QQ]$,
		and compute a $\ZZ$-basis
		$\gamma_1,\ldots,\gamma_e$ of $\integers{K}$.
\item[3] 	For each $1 \le i \le e$,
		let $\Gamma^{(i)}$
		be the $e\times e$ integer matrix 
		such that
		\[
			\gamma_i\gamma_j
			=
			\sum_{k=1}^e \Gamma^{(i)}_{jk}\gamma_k
			\quad 
			\text{for all}
			\ 1 \le j \le e ,
		\]
		and let
		$\rho: \integers{K} \to \Mat_{e\times e}(\ZZ)$
		be the map 
		$\sum_{i=1}^e a_i\gamma_i \mapsto \sum_{i=1}^ea_i\Gamma^{(i)}$.
\item[4] 	For each $1 \le k \le n-1$,
		\begin{description}
		\item[4a]	Let $M$ be the $2ek\times2ek$ block matrix
				\[
					M 
					:= 
					\left(\rho(\differentialblock{A}{n-1}_{i,j})\right)_{i,j=1}^k
					\oplus
					\left(\rho(\differentialblock{A}{n-1}^\sigma_{i,j})\right)_{i,j=1}^k
					.
				\]
		\item[4b]	Compute the Hermite Normal Form of $M$,
				and 
				let $(d_1,\ldots,d_{2ek})$ be
				its elementary divisors.
		\item[4c]	Set 
				$G(A,k) := \prod_{i=1}^{2ek}(\ZZ/d_i\ZZ)$.
		\end{description}
\item[5] 	Return $(G(A,1),\ldots,G(A,n-1))$.
\end{description}
\end{algorithm}

\begin{remark}
	In our examples, the generic fibres $X$, $Y$, and $C$
	are defined over $K(s_2,\ldots,s_d)$ or $K(s_2,\ldots,s_d,t)$,
	where $K$ is a number field.
	But if $Q_X$ and $Q_Y$ are defined over $K$
	then so is~$A$, 
	so we can apply Algorithm~\ref{algorithm:Gak}
	and use Eq.~\eqref{eq:kernel-relation}
	to deduce the structure of $\ker\phi_C$
	without choosing any particular specialization.
\end{remark}

\section{
	Pairs of polynomials
}
\label{sec:polynomials}

To produce nontrivial examples
in the form of Eq.~\eqref{eq:setup},
we need a source of
pairs of polynomials $(Q_X,Q_Y)$
such that $Q_X(x_1) - Q_Y(x_2)$ is reducible.
For indecomposable~$Q_X$ and~$Q_Y$ 
over~$\CC$,
these pairs 
have been explicitly classified
by Cassou--Nogu\`es and Couveignes~\cite{Cassou-Nogues--Couveignes}.
The pairs
are deeply interesting in their own right:
For further background,
we refer to the work of
Cassels~\cite{Cassels},
Davenport, Lewis, and Schinzel~\cite{Davenport--Lewis--Schinzel,Davenport--Schinzel},
Feit~\cite{Feit-1,Feit-2,Feit-3},
and Fried~\cite{Fried-1,Fried-2,Fried-3}.
An excellent account of the context and importance of these results
can be found on Fried's website~\cite{Fried-site}.
The plane curves cut out by the factors themselves
are also interesting; Avanzi's thesis~\cite{Avanzi-thesis}
provides a good introduction to this topic.

\begin{definition}
\label{def:equivalence}
	We say that polynomials $f_1$ and $f_2$ over $K$ are
	\emph{linear translates}
	if~$f_1(x) = f_2(ax + b)$
	for some $a$, $b$ in $\closure{K}$ with $a$ nonzero.
	We say pairs 
	of polynomials 
	$(f_1,g_1)$ and $(f_2,g_2)$ 
	are \emph{equivalent}
	if there exists some $a$, $b$ in $\closure{K}$
	with $a$ nonzero
	such that $f_1$ and $af_2 + b$ are linear translates
	and $g_1$ and $ag_2 + b$ are linear translates.
\end{definition}

The ``equivalence'' of Definition~\ref{def:equivalence}
is indeed an equivalence relation
on pairs of polynomials.
From the point of view of constructing homomorphisms,
equivalent pairs of polynomials 
give rise to isomorphic homomorphisms of Jacobians.

\begin{proposition}
\label{proposition:equivalence-classes}
	Let $\family{X}$,
	$\family{Y}$,
	and $\family{C}$
	be as in Eq.~\eqref{eq:setup}.
	Suppose that $(Q_Z,Q_W)$ is equivalent to $(Q_X,Q_Y)$
	(so $Q_Z(x) = aQ_X(a_1x + b_1) + b$
	and $Q_W(x) = aQ_Y(a_2x + b_2) + b$
	for some $a$, $b$, $a_1$, $b_1$, $a_2$, and $b_2$ in $\closure{K}$
	with $a$, $a_1$, and $a_2$ nonzero.)
	\begin{enumerate}
	\item	If $A(x_1,x_2)$ is a $\closure{K}$-irreducible factor 
		of $Q_X(x_1)-Q_Y(x_2)$,
		then
		$A'(x_1,x_2) = A(a_1x_1 + b_1,a_2x_2 + b_2)$
		is a $\closure{K}$-irreducible factor 
		of $Q_Z(x_1)-Q_W(x_2)$.
	\item	$(\family{X},\family{Y})$
		is $\closure{K}$-isomorphic
		to $(\family{W}:P_d(y_1)=Q_W(x_1),\family{Z}:P_d(y_2)=Q_Z(x_2))$,
		and $\phi_{\family{C}}$
		is $\closure{K}$-isomorphic to $\phi_{\family{D}}$
		where $\family{D} = \variety{y_1-y_2,A'(x_1,x_2)}\subset\family{W}\times\family{Z}$.
	\end{enumerate}
\end{proposition}
\begin{proof}
	Part (1) is a straightforward symbolic exercise.
	For Part (2),
	let~$\alpha := a^{-1/d}$.
	The family 
	$(\family{Z}, \family{W})$
	is $\closure{K}$-isomorphic to $(\family{X},\family{Y})$ via
	\[
	\begin{array}{r@{\;\longmapsto\;}l}
		(s_2,\ldots,s_d) 
		& 
		(\alpha^2 s_2,\ldots, \alpha^{d-1}s_{d-1}, \alpha^d s_d - b/a), 
		\\
		(x_i,y_i) 
		& 
		(a_ix_i+b_i,\alpha y_i) .
	\end{array}
	\]
	This induces
	a $\closure{K}$-isomorphism between $\family{C}$ and $\family{D}$,
	so $\phi_{\family{C}} \cong \phi_{\family{D}}$.
\end{proof}

The classification of pairs of indecomposable polynomials $(Q_X,Q_Y)$
over $\CC$
such that $Q_X(x_1)-Q_Y(x_2)$
has a nontrivial factor
splits naturally into two parts,
according to whether $Q_X$ and~$Q_Y$ are linear translates or not.
Observe that if $Q_X$ and $Q_Y$ are linear translates,
then by Proposition~\ref{proposition:equivalence-classes}(1)
we reduce to the case $Q_Y = Q_X$.
We always have a factor $x_1 - x_2$ of $Q_X(x_1)-Q_X(x_2)$;
this corresponds to the fact that the endomorphism ring of $\Jac{X}$
always contains $\ZZ$ (cf.~Example~\ref{example:diagonal}).

\begin{theorem}[Fried~\cite{Fried-1}]
\label{theorem:linear-translate-theorem}
	Let $Q_X$ be an indecomposable polynomial
	of degree at least~$3$ over $\CC$.
	Then 
	$(Q_X(x_1) - Q_X(x_2))/(x_1-x_2)$ is $\closure{K}$-reducible
	if and only if $(Q_X,Q_X)$ is equivalent to either 
	\begin{enumerate}
		\item	the pair $(x^n,x^n)$
			for some odd prime $n$, or
		\item	the pair $(D_n(x,1),D_n(x,1))$
			for some odd prime $n$,
			where $D_n(x,1)$ is 
			the $n^\mathrm{th}$ Dickson polynomial	
			of the first kind with parameter $1$
			(see Remark~\ref{remark:Dickson}).
	\end{enumerate}
\end{theorem}

\begin{theorem}[Cassou--Nogu\`es and Couveignes~\cite{Cassou-Nogues--Couveignes}]
\label{theorem:CNC-theorem}
	Let $(Q_X,Q_Y)$ be indecomposable polynomials 
	of degree at least~$3$ over $\CC$,
	and let $\sigma$ denote complex conjugation.
	Assume the classification of finite simple groups
	(see Remark~\ref{remark:CFSG}).
	If $Q_X$ and $Q_Y$ are not linear translates,
	then $Q_X(x_1) - Q_Y(x_2)$ is reducible
	if and only if 
	$(Q_X,Q_Y)$ is equivalent 
	(possibly after exchanging $Q_X$ and $Q_Y$)
	to
	\begin{enumerate}
		\item	a pair in the 
			one-parameter family $(f_{7},f_{7}^\sigma)$
			defined in \S\ref{family:degree-7},
			or
		\item	the pair $(f_{11},f_{11}^\sigma)$
			defined in \S\ref{family:degree-11},
			or
		\item	a pair in the
			one-parameter family $(f_{13},f_{13}^\sigma)$
			defined in \S\ref{family:degree-13},~or
		\item	a pair in the
			one-parameter family $(f_{15},-f_{15}^\sigma)$
			defined in~\S\ref{family:degree-15}, or
		\item	the pair $(f_{21},f_{21}^\sigma)$
			defined in \S\ref{family:degree-21},
			or
		\item	the pair $(f_{31},f_{31}^\sigma)$
			defined in \S\ref{family:degree-31}.
	\end{enumerate}
\end{theorem}

It follows from Proposition~\ref{proposition:equivalence-classes}
that we can give a complete treatment 
of homomorphisms induced by
correspondences in the form of Eq.~\eqref{eq:setup}
by applying our constructions to the polynomials 
of
Theorems~\ref{theorem:linear-translate-theorem}
and~\ref{theorem:CNC-theorem}.
We treat $x^n$ and $D_n(x,1)$ 
in~\S\ref{section:endomorphisms},
and the polynomials
$f_7$, $f_{11}$, $f_{13}$, $f_{15}$, $f_{21}$, and $f_{31}$
from Theorem~\ref{theorem:CNC-theorem}
in \S\S\ref{family:degree-7}-\ref{family:degree-31}.

\begin{remark}
\label{remark:indecomposability}
	The restriction to indecomposable polynomials
	is not too heavy,
	since we are primarily interested in
	isogenies of absolutely simple Jacobians.
	If $Q_X(x) = Q_1(Q_2(x))$
	with $\deg Q_2 > 1$,
	then
	we have 
	a $(\deg Q_2)$-uple cover $(x,y)\mapsto(Q_2(x),y)$
	from $\family{X}$
	to $\family{X}': P_d(y) = Q_1(x)$.
	If $d > 2$ and $\deg Q_1 > 1$,
	or if $d = 2$ and $\deg Q_1 > 2$,
	then~$\family{X}'$ has positive genus
	and~$\Jacfamily{\family{X}'}$ 
	is a nontrivial isogeny factor of $\Jacfamily{\family{X}}$,
	so~$\Jacfamily{\family{X}}$ is reducible.
	If $d = \deg Q_1 = 2$,
	then $\Jacfamily{\family{X}}$ is not necessarily reducible:
	a partial treatment of 
	this case appears as the \emph{quadratic construction}
	in~\cite{Smith-hyperelliptic}.
\end{remark}

\begin{remark}
\label{remark:CFSG}
	Theorem~\ref{theorem:CNC-theorem}
	assumes the classification of finite simple groups~\cite{Gorenstein--Lyons--Solomon}.
	The classification is only required 
	to prove the completeness of 
	the list of pairs of polynomials 
	(and not for the existence
	of the factorizations).  
	In particular,
	Theorem~\ref{theorem:main}
	does not depend on
	the classification of finite simple groups;
	but 
	one corollary 
	of the classification
	is that 
	every isogeny induced by a correspondence 
	in the form of~Eq.~\eqref{eq:setup}
	is isomorphic to
	a composition of endomorphisms from the families in \S\ref{section:endomorphisms}
	and isogenies from the families in Theorem~\ref{theorem:main}.
\end{remark}

\begin{remark}
\label{remark:Dickson}
	Recall that $D_n(x,a)$ is
	the $n^\mathrm{th}$ Dickson polynomial of the first kind with parameter~$a$
	(see~\cite{Lidl-Mullen-Turnwald}):
	that is, the unique polynomial of degree $n$
	such that $D_n(x + a/x,a) = x^n + (a/x)^n$.
	In characteristic zero
	$D_n(x,1) = 2T_n(x/2)$,
	where $T_n$ is the~$n^\mathrm{th}$ classical Chebyshev polynomial.
	We have $D_n(x,a) = a^{n/2}D_n(a^{-1/2}x,1)$
	when $a \not= 0$,
	so $(D_n(x,a),D_n(x,a))$
	is equivalent to $(D_n(x,1),D_n(x,1))$.
	On the other hand $D_n(x,0) = x^n$,
	so Theorem~\ref{theorem:linear-translate-theorem}(1)
	is essentially a specialization of 
	Theorem~\ref{theorem:linear-translate-theorem}(2).
\end{remark}

\section{
	Families with explicit Complex and Real Multiplication
}
\label{section:endomorphisms}

\label{family:cyclotomic-CM}

We now put our techniques into practice.
First,
consider Theorem~\ref{theorem:linear-translate-theorem}(1):
Let $Q_X(x) = Q_Y(x) = x^n$
for some odd prime~$n$.
For each $d > 1$,
we derive a family
\[
	\family{Z}_{d,n}: P_d(y) = x^n
\]
of curves of genus $\genusfn{d}{n}$
with an automorphism
\( \zeta : (x,y) \mapsto (\zeta_n x, y) \) of order~$n$.
We say $\family{Z}_{d,n}$ is \emph{superelliptic}
if $n\nmid d$. 
The family has $d-2$ moduli:
restricting the isomorphisms of \S\ref{sec:moduli} to $\family{Z}_{d,n}$,
we see that every isomorphism class in $\family{Z}_{d,n}$
contains a unique representative with $s_2 = 1$.
We identify $\zeta$ with its induced endomorphism
of~$\Jacfamily{\family{Z}_{d,n}}$;
its minimal polynomial is the $n^\mathrm{th}$ cyclotomic polynomial.
Recalling that
\[
	x_1^n - x_2^n = \prod_{i=0}^{n-1}(\zeta_n^ix_1 - x_2) 
	,
\]
we consider the correspondences
\[ 
	\family{C}_i := \variety{y_1-y_2,\zeta_n^ix_1-x_2 } 
	\subset 
	\product[\QQ(\zeta_n)(s_2,\ldots,s_d)]{\family{Z}_{d,n}}{\family{Z}_{d,n}}
	.
\]
We have 
$\family{C}_i = (\product{\mathrm{Id}}{\zeta^i})(\family{Z}_{d,n})$
so $\phi_{\family{C}_i} = \zeta^i$
(cf.~Example~\ref{example:diagonal});
the 
\(
	\family{C}_i 
\)
therefore generate a subring of $\End(\Jacfamily{\family{Z}_{d,n}})$
isomorphic to $\ZZ[\zeta_n]$.

Now consider Theorem~\ref{theorem:linear-translate-theorem}(2):
$Q_X(x) = Q_Y(x) = D_n(x,1)$
for some odd prime $n$.
For each $d > 1$, 
we derive a family
\[
	\family{W}_{d,n}: P_d(y_i) = D_n(x_i,1) 
\]
of curves of genus $\genusfn{d}{n}$
with $d-1$ moduli. 
In~\cite[Theorem~3.12]{Lidl-Mullen-Turnwald}
we see that
\[
	D_n(x_1,1) - D_n(x_2,1)
	=
	(x_1 - x_2)
	\!\!\!\!\prod_{i=1}^{(n-1)/2}\!\!\!\!
	A_{n,i}(x_1,x_2)
	,
\]
where
\[
	A_{n,i}(x_1,x_2) 
	:= 
	x_1^2 + x_2^2 
	- (\zeta_n^i + \zeta_n^{-i})x_1x_2 
	+ (\zeta_n^i - \zeta_n^{-i})^2
	.
\]

\begin{proposition} 
\label{proposition:RM-family}
	The endomorphisms of $\Jacfamily{\family{W}_{d,n}}$
	induced by the correspondences
	\[ 
		C_i := \variety{y_1-y_2, A_{n,i}(x_1,x_2) } 
		\subset 
		\product[\QQ(\zeta_n)(s_2,\ldots,s_d)]{\family{W}_{d,n}}{\family{W}_{d,n}}
	\]
	generate a subring of $\End(\Jacfamily{\family{W}_{d,n}})$
	isomorphic to $\ZZ[\zeta_n + \zeta_n^{-1}]$.
\end{proposition}
\begin{proof}
	The family $\family{U}_{d,n}: P_d(v) = u^n + 1/u^n$
	has
	an involution $\iota: (u,v) \mapsto (1/u,v)$
	and an automorphism $\zeta: (u,v) \mapsto (\zeta_n u, v)$.
	The double cover
	$\pi: \family{U}_{d,n} \to \family{W}_{d,n}$
	defined by
	$(u,v) \mapsto (u + u^{-1},v)$
	is 
	the quotient of $\family{U}_{d,n}$
	by $\subgrp{\iota}$,
	and
	$\pi_*\pi^* = \multiplication[\Jacfamily{\family{W}_{d,n}}]{2}$.
	Let $(x,y)$ be a generic point on $\family{W}_{d,n}$.
	On the level of divisors we have
	\[
		\phi_{C_i}((x,y)) = (\alpha_1, y) + (\alpha_2, y) ,
	\]
	where $\alpha_1+\alpha_2 = (\zeta_n^{i} + \zeta_n^{-i})x$
	and $\alpha_1\alpha_2 = x^2 + (\zeta_n^{i} - \zeta_n^{-i})^2$.
	On the other hand, 
	\[
		\pi_*(\zeta^{i} + \zeta^{-i})\pi^*((x,y)) 	
		=
		2(\zeta_n^{i}\beta+\zeta_n^{-i}\beta^{-1},y)
			+ 2(\zeta_n^{-i}\beta+\zeta_n^{i}\beta^{-1},y) ,
	\]
	where $\beta + \beta^{-1} = x$.
	But
	\(
		\{
		\zeta_n^{i}\beta+\zeta_n^{-i}\beta^{-1}
		, 
		\zeta_n^{-i}\beta+\zeta_n^{i}\beta^{-1}
		\}
		=
		\{\alpha_1 , \alpha_2\}
	\),
	so $\pi_*(\zeta^{i} + \zeta^{-i})\pi^*((x,y)) = 2\phi_{C_i}((x,y))$,
	and hence
	\[
		\pi_*(\zeta^i + \zeta^{-i})\pi^* = [2]\phi_{C_i} .
	\]
	Let $m_i$ be the minimal polynomial of $\zeta_n^i + \zeta_n^{-i}$;
	it is irreducible, and $m_i(\zeta^i + \zeta^{-i}) = 0$. 
	Working in $\QQ(\phi_{C_i})$,
	we have 
	\[
		2m_i(\phi_{C_i}) 
		=
		2m_i(\frac{1}{2}\pi_*(\zeta^i + \zeta^{-i})\pi^*)
		=
		\pi_*m_i(\zeta^i + \zeta^{-i})\pi^*
		=
		0
		;
	\]
	hence $m_i(\phi_{C_i}) = 0$,
	and the proposition follows.
\end{proof}

\begin{remark}
	The family $\family{W}_{2,n}$
	is isomorphic to the family $\family{C}_t$
	of
	hyperelliptic curves of genus $(n-1)/2$
	described by Tautz, Top, and Verberkmoes~\cite{Tautz--Top--Verberkmoes}.
	Their families
	extend earlier families of Mestre~\cite{Mestre--RM},
	replacing subgroups of the $n$-torsion of elliptic curves 
	with the group of $n^\mathrm{th}$ roots of unity in $\QQbar$.
	Our construction of $\family{W}_{d,n}$
	readily generalizes in the other direction
	to give more families of Jacobians
	in genus~$\genusfn{d}{n}$
	with Real Multiplication 
	by~$\ZZ[\zeta_n + \zeta_n^{-1}]$
	(though for these families,
	the Dickson polynomials are replaced by certain rational functions).
\end{remark}

\section{
	Genus $\genusfn{d}{7}$ families 
	from Theorem~\ref{theorem:CNC-theorem} (1)
}
\label{family:degree-7}

Consider Theorem~\ref{theorem:CNC-theorem}(1):
Let~$\alpha_7$ be an element of $\QQbar$
satisfying
\[
	\alpha_7^2 + \alpha_7 + 2 = 0 ,
\]
so $\QQ(\alpha_7) = \QQ(\sqrt{-7})$.
The involution $\sigma: \alpha_7 \mapsto 2/\alpha_7$ 
generates $\Gal{\QQ(\alpha_7)/\QQ}$.
Let~$t$ be a free parameter,
and let $f_{7}$ be the polynomial 
over $\QQ(\alpha_7)[t]$
defined by
\[
	\begin{array}{r@{\;}l}
		f_{7}(x) 
		:= 
		&
		x^7 
		- 7\alpha_7t x^5 
		- 7\alpha_7t x^4 
		- 7(2 \alpha_7 + 5)t^2 x^3 
		- 7(4 \alpha_7 + 6)t^2 x^2 
		\\
		&
		{} + 7((3 \alpha_7 - 2)t^3 - (\alpha_7 + 3)t^2) x 
		+ 7\alpha_7t^3
	\end{array}
\]
(so $f_{7} = 7g$,
where $g$ is
the polynomial of~\cite[\S5.1]{Cassou-Nogues--Couveignes}
with $a_2 = \alpha_7$ and $T = t$).
We have a factorization 
$f_{7}(x_1) - f_{7}^\sigma(x_2) = A_7(x_1,x_2)B_7(x_1,x_2)$,
where
\[
	A_{7}
	= 
	x_1^3 - x_2^3 - \alpha_7^\sigma x_1^2x_2 + \alpha_7 x_1x_2^2 
	+ (3-2\alpha_7^\sigma)tx_1 - (3-2\alpha_7)tx_2 + (\alpha_7 - \alpha_7^\sigma)t
	.
\]
Both $A_{7}$ and $B_{7}$ 
are absolutely irreducible,
and
$\tau(A_{7}) = -A_{7}^\sigma$ and $\tau(B_{7}) = B_{7}^\sigma$.


\begin{proposition}
\label{proposition:phi-d-7}
	Let $d > 1$ be an integer, 
	and consider 
	the families
	defined by
	\[
	\begin{array}{c}
		\family{X}_{d,7} : P_d(y_1) = f_{7}(x_1) 
		,\quad 
		\family{Y}_{d,7} : P_d(y_2) = f_{7}^\sigma(x_2) 
		,\medskip \\
		\family{C}_{d,7} = \variety{y_1-y_2,A_{7}(x_1,x_2)}
		\subset
		\product[\QQ(\alpha_7)(s_2,\ldots,s_d,t)]{\family{X}_{d,7}}{\family{Y}_{d,7}}
		.
	\end{array}
	\]
	The induced homomorphism
	\(
		\phi_{d,7}
		=
		\phi_{\family{C}_{d,7}}: 
		\Jacfamily{\family{X}_{d,7}} 
		\to~\Jacfamily{\family{Y}_{d,7}}
	\)
	is a $d$-dimensional family
	of $\isogenytype{2}{\genusfn{d}{7}}$-isogenies.
\end{proposition}
\begin{proof}
	Both $\family{X}_{d,7}$ and $\family{Y}_{d,7}$
	have genus $\genusfn{d}{7}$,
	with $d$ moduli by
	Lemma~\ref{lemma:family-dimension-with-t}.
	Applying Algorithm~\ref{algorithm:differential-matrix-block} to $A_7$,
	we find
	that \( \differentialblock{A_{7}}{6} \) is equal to
	\[
		\left(\!\!\!\begin{array}{rrrrrr}
\alpha_{7} & 0 & 0 & 0 & 0 & 0 \\
0 & \alpha_{7} & 0 & 0 & 0 & 0 \\
-3(2\alpha_{7} + 1) t & 0 & \alpha_{7}^\sigma & 0 & 0 & 0 \\
-4(\alpha_{7} + 4) t & -4(\alpha_{7} + 4) t & 0 & \alpha_{7} & 0 & 0 \\
35(\alpha_{7} + 2) t^2 & -5(2\alpha_{7} + 1) t & -5(\alpha_{7} - 3) t & 0 & \alpha_{7}^\sigma & 0 \\
-42(\alpha_{7} - 3) t^2 & -21(2 \alpha_{7} - 3) t^2 & -6(\alpha_{7} - 3) t & -6(2 \alpha_{7} + 1) t & 0 & \alpha_{7}^\sigma
		\end{array}\!\right)
		.
	\] 
	We have
	\[
		\differentialblock{A_{7}}{6}
		\differentialblock{\tau(A_{7})}{6}
		=
		\differentialblock{A_{7}}{6}
		\differentialblock{-A_{7}^\sigma}{6}
		=
		\differentialblock{A_{7}}{6}
		\differentialblock{A_{7}}{6}^\sigma
		=
		2I_{6}
	\]
	(since $\tau(A_{7}) = -A_{7}^\sigma$),
	so $\phi_{d,7}$ is 
	a family of $\isogenytype{2}{\genusfn{d}{7}}$-isogenies
	by
	Lemma~\ref{lemma:isogeny-classification}.
\end{proof}

\begin{remark}
	We may view $\phi_{d,7}$
	as a deformation of an endomorphism of 
	the superelliptic Jacobian $\Jacfamily{\family{Z}_{d,7}}$
	of \S\ref{family:cyclotomic-CM}.
	Embed $\ZZ[\alpha_7]$ in~$\ZZ[\zeta_7]$,
	identifying~$\alpha_7$ with $\zeta_7 + \zeta_7^2 + \zeta_7^4$.
	At $t = 0$,
	both $\family{X}_{d,7}$ and $\family{Y}_{d,7}$
	specialize to $\family{Z}_{d,7}$,
	which
	has an automorphism $\zeta:(x,y)\mapsto(\zeta_7 x,y)$ of order $7$,
	while $\family{C}_{d,7}$
	specializes~to 
	\[
		\begin{array}{r@{\;=\;}l}
			C_0 
			&
			\variety{y_1-y_2,
				x_1^3 - x_1^2x_2 + \alpha_{7}^\sigma x_1x_2^2 - x_2^3
			}
			\\ &
			\sum_{i\in\{1,2,4\}}
			\variety{y_1-y_2,\zeta_7^ix_1 - x_2}
			\subset 
			\product[\QQ(\alpha_7)(s_2,\ldots,s_d)]{\family{Z}_{d,7}}{\family{Z}_{d,7}}
			.
		\end{array}
	\]
	Each 
	$\variety{y_1-y_2,\zeta_7^ix_1 - x_2}$ 
	induces $\zeta^i$ on~$\Jacfamily{\family{Z}_{d,7}}$,
	so 
	\(
		\phi_{C_0} 
		= 
		\zeta + \zeta^2 + \zeta^4 
		=
		\multiplication[\Jacfamily{\family{Z}_{d,7}}]{\alpha_{7}}
	\).
	Therefore,~$\phi_{d,7}$
	is a one-parameter deformation of 
	\( 
		\multiplication[\Jacfamily{\family{Z}_{d,7}}]{\alpha_{7}}
	\),
	which splits
	\(
		\multiplication[\Jacfamily{\family{Z}_{d,7}}]{2}
	\).
	(This gives
	an alternative proof of 
	Proposition~\ref{proposition:phi-d-7}.)
\end{remark}

\begin{remark}
	Given any hyperelliptic curve $X$ of genus~$3$
	and a maximal $2$-Weil isotropic subgroup $S$ of $\Jac{X}[2]$,
	there exists a (possibly reducible, and generally non-hyperelliptic) curve $Y$
	of genus~$3$ and
	an isogeny $\phi: \Jac{X} \to \Jac{Y}$
	with kernel~$S$,
	which may be defined 
	over a quadratic extension of $K(S)$.
	An algorithm to compute equations for $Y$ and $\phi$
	when $S$ is generated by differences of Weierstrass points
	appears in~\cite{Smith-genus-3}.
	Mestre~\cite{Mestre} gives a $4$-parameter family 
	of $\isogenytype{2}{3}$-isogenies
	of hyperelliptic Jacobians;
	their
	kernels are also generated by differences of
	Weierstrass points.
	Since~$\Jacfamily{\family{X}_{2,7}}[2]$
	is generated by differences of Weierstrass points,
	which correspond to roots of $f_{7}$ together with the point at infinity,
	we can factor $f_{7}$ (or a reduction at
	some well-chosen prime) over its splitting field,
	and then explicitly compute the restriction of~$\phi_{2,7}$
	to $\Jacfamily{\family{X}_{2,7}}[2]$
	to show that its kernel 
	is \emph{not} generated by differences of Weierstrass points.
	Therefore,~$\phi_{2,7}$ is not one of the isogenies
	of~\cite{Smith-genus-3}
	or~\cite{Mestre}.
\end{remark}

\section{
	Genus $\genusfn{d}{11}$ families 
	from Theorem~\ref{theorem:CNC-theorem} (2)
}
\label{family:degree-11}

Consider Theorem~\ref{theorem:CNC-theorem}(2):
Let~$\alpha_{11}$
be an element of $\QQbar$
satisfying
\[
	\alpha_{11}^2 + \alpha_{11} + 3 = 0 ,
\]
so $\QQ(\alpha_{11}) = \QQ(\sqrt{-11})$.
The involution $\sigma: \alpha_{11} \mapsto 3/\alpha_{11}$ 
generates $\Gal{\QQ(\alpha_{11})/\QQ}$.
Let
$f_{11}$ be the polynomial 
over~$\QQ(\alpha_{11})$
defined by
\[
	\begin{array}{r@{\;}l}
	f_{11}(x) := 
	&
	x^{11} + 11\alpha_{11} x^9 + 22 x^8 - 33(\alpha_{11} + 4) x^7 + 176 \alpha_{11} x^6 
	\\
	&
	{} - 33(7\alpha_{11} - 5) x^5 - 330(\alpha_{11} + 4) x^4 + 693 (\alpha_{11} + 1) x^3 
	\\
	&
	{} - 220(5\alpha_{11} - 1) x^2 - 33(8\alpha_{11} + 47) x + 198 \alpha_{11}
	\end{array}
\]
(so $f_{11} = 11g$,
where $g$ is 
the polynomial of~\cite[\S5.2]{Cassou-Nogues--Couveignes}
with $a_2 = \alpha_{11}^\sigma$).
We have a factorization
$f_{11}(x_1) - f_{11}^\sigma(x_2) = A_{11}(x_1,x_2)B_{11}(x_1,x_2)$,
where 
\[
	\begin{array}{r@{\;}l}
	A_{11}(x_1,x_2) 
	=
	&
	x_1^5 - \alpha_{11}x_1^4x_2 - x_1^3x_2^2 + (4 \alpha_{11} + 2)x_1^3 + x_1^2x_2^3 + (\alpha_{11} + 6)x_1^2x_2 
	\\ & {}
	- (2 \alpha_{11} - 10)x_1^2 - (\alpha_{11} + 1)x_1x_2^4 + (\alpha_{11} - 5)x_1x_2^2 
	\\ & {}
	- (12 \alpha_{11} + 6)x_1x_2 + (8 \alpha_{11} - 7)x_1 - x_2^5 + (4 \alpha_{11} + 2)x_2^3 
	\\ & {}
	- (2 \alpha_{11} + 12)x_2^2 + (8 \alpha_{11} + 15)x_2 + 12 \alpha_{11} + 6
	.
	\end{array}
\]
Both $A_{11}$ and $B_{11}$ are absolutely irreducible,
and~$\tau(A_{11}) = -A_{11}^\sigma$
and~$\tau(B_{11}) = B_{11}^\sigma$.

\begin{proposition}
\label{proposition:phi-d-11}
	Let $d > 1$ be an integer,
	and consider the families
	defined by
	\[
	\begin{array}{c}
		\family{X}_{d,11} : P_d(y_1) = f_{11}(x_1) 
		,\quad 
		\family{Y}_{d,11} : P_d(y_2) = f_{11}^\sigma(x_2) 
		,\medskip \\
		\family{C}_{d,11}
		=
		\variety{y_1-y_2,A_{11}(x_1,x_2)}
		\subset
		\product[\QQ(\alpha_{11})(s_2,\ldots,s_d)]{\family{X}_{d,11}}{\family{Y}_{d,11}}
		.
	\end{array}
	\]
	The induced homomorphism
	\(
		\phi_{d,11}
		=
		\phi_{\family{C}_{d,11}}: 
		\Jacfamily{\family{X}_{d,11}} 
		\to \Jacfamily{\family{Y}_{d,11}}
	\)
	is a $(d-1)$-dimensional family of $\isogenytype{3}{\genusfn{d}{11}}$-isogenies.
\end{proposition}
\begin{proof}
	Both $\family{X}_{d,11}$ and $\family{Y}_{d,11}$
	have genus $\genusfn{d}{11}$,
	and $d-1$ moduli by Lemma~\ref{lemma:family-dimension-no-t}.
	As in Proposition~\ref{proposition:phi-d-7},
	we calculate
	$\differentialblock{A_{11}}{10}$
	(given in~\texttt{degree-11.m})
	using Algorithm~\ref{algorithm:differential-matrix-block};
	its diagonal entries
	are all either $\alpha_{11}$ or $\alpha_{11}^\sigma$.
	Using $\tau(A_{11}) = -A_{11}^\sigma$,
	we find
	\[
		\differentialblock{A_{11}}{10}
		\differentialblock{\tau(A_{11})}{10}
		=
		\differentialblock{A_{11}}{10}
		\differentialblock{A_{11}}{10}^\sigma
		=
		3I_{10}
		,
	\]
	so $\phi_{d,11}$ 
	is a family of $\isogenytype{3}{\genusfn{d}{11}}$-isogenies
	by Lemma~\ref{lemma:isogeny-classification}.
\end{proof}

\section{
	Genus $\genusfn{d}{13}$ families 
	from Theorem~\ref{theorem:CNC-theorem} (3)
}
\label{family:degree-13}

Consider Theorem~\ref{theorem:CNC-theorem}(3):
Let~$\beta_{13}$
and $\alpha_{13}$
be elements of $\QQbar$ satisfying
\[
	\beta_{13}^2 - 5\beta_{13} + 3 = 0 
	\text{\quad and\quad }
	\alpha_{13}^2 + (\beta_{13}-2)\alpha_{13} + \beta_{13} = 0 .
\]
The field
$\QQ(\alpha_{13}) = \QQ(\sqrt{-3\sqrt{13}+1})$ 
is an imaginary quadratic extension of the real quadratic field
$\QQ(\beta_{13}) = \QQ(\sqrt{13})$.
The involution
$\sigma: \alpha_{13} \mapsto \beta_{13}/\alpha_{13}$ 
generates $\Gal{\QQ(\alpha_{13})/\QQ(\beta_{13})}$.
Let $t$ be a free parameter, and~let
\[
	f_{13}(x) 
	=
	x^{13}
	+
	39((3 \beta_{13} - 13) \alpha_{13} - 2 \beta_{13} + 8) t x^{11}
	+ \cdots
\]
be the polynomial of degree $13$ over $\QQ(\alpha_{13})[t]$
defined in the file \texttt{degree-13.m}
(we have $f_{13} = 13g$,
where $g$ is the polynomial 
of~\cite[\S5.3]{Cassou-Nogues--Couveignes} with $a_1 = \alpha_{13}$ and $T = t$).
We have a factorization
\(
	f_{13}(x_1) - f_{13}^\sigma(x_2) 
	= 
	A_{13}(x_1,x_2)B_{13}(x_1,x_2)
\),
where
\[
	\begin{array}{r@{\;}l}
	A_{13}(x_1,x_2)
	= 
	&
	x_1^4  + x_2^4
	+ (\beta_{13} - 3) x_1^2 x_2^2
	- 9(3 \beta_{13} - 14) t x_1 x_2
	+ 12(47 \beta_{13} - 202) t^2
	\\ & {} 
	- ((\beta_{13} - 4) \alpha_{13} + 2) x_1^3 x_2
	+ ((\beta_{13} - 4) \alpha_{13} - \beta_{13} + 3) x_1 x_2^3
	\\ & {} 
	+ 3((17 \beta_{13} - 73) \alpha_{13} - 12 \beta_{13} + 50) t x_1^2
	\\ & {} 
	- 3((17 \beta_{13} - 73) \alpha_{13} - 10 \beta_{13} + 45) t x_2^2
	\\ & {} 
	+ 3((5 \beta_{13} - 22) \alpha_{13} - 9 \beta_{13} + 38) t x_1
	\\ & {} 
	- 3((5 \beta_{13} - 22) \alpha_{13} + 2 \beta_{13} - 9) t x_2
	.
	\end{array}
\]
Both $A_{13}$ and $B_{13}$ are absolutely irreducible,
and~$\tau(A_{13}) = A_{13}^\sigma$
and~$\tau(B_{13}) = -B_{13}^\sigma$.

\begin{proposition}
\label{proposition:phi-d-13}
	Let $d > 1$ be an integer,
	and consider the families
	defined by
	\[
	\begin{array}{c}
		\family{X}_{d,13} : P_d(y_1) = f_{13}(x_1) 
		,\quad  
		\family{Y}_{d,13} : P_d(y_2) = f_{13}^\sigma(x_2) 
		,\medskip \\
		\family{C}_{d,13} = \variety{y_1-y_2,A_{13}(x_1,x_2)}
		\subset
		\product[\QQ(\alpha_{13})(s_2,\ldots,s_d,t)]{\family{X}_{d,13}}{\family{Y}_{d,13}}
		.
	\end{array}
	\]
	The induced homomorphism
	\(
		\phi_{d,13}
		:= 
		\phi_{\family{C}_{d,13}}:
		\Jacfamily{\family{X}_{d,13}} 
		\to~\Jacfamily{\family{Y}_{d,13}}
	\) 
	is a $d$-dimensional family
	of $\isogenytype{3}{\genusfn{d}{13}}$-isogenies.
\end{proposition}
\begin{proof}
	Both $\family{X}_{d,13}$ and $\family{Y}_{d,13}$
	have genus $\genusfn{d}{13}$,
	with $d$ moduli by Lemma~\ref{lemma:family-dimension-with-t}.
	We compute
	$\differentialblock{A_{13}}{12}$
	(given in~\texttt{degree-13.m})
	using Algorithm~\ref{algorithm:differential-matrix-block};
	its diagonal is
	\[
		( 
		\lambda_1, 
		\lambda_2, 
		\lambda_1, 
		\lambda_1^\sigma,  
		\lambda_2,  
		\lambda_2,  
		\lambda_2^\sigma,  
		\lambda_2^\sigma,  
		\lambda_1,  
		\lambda_1^\sigma,  
		\lambda_2^\sigma,  
		\lambda_1^\sigma 
		),
	\]
	where 
	$\lambda_1 = (\beta_{13} - 4)\alpha_{13} + 2$
	and $\lambda_2 = \alpha_{13} + 1$
	both have norm $3$
	in~$\QQ(\beta_{13})$.
	We find
	\[
		\differentialblock{A_{13}}{12}
		\differentialblock{\tau(A_{13})}{12}
		=
		\differentialblock{A_{13}}{12}
		\differentialblock{A_{13}}{12}^\sigma
		=
		3I_{12}
	\]
	(since $\tau(A_{13}) = A_{13}^\sigma$),
	so
	the result follows from 
	Lemma~\ref{lemma:isogeny-classification}.
\end{proof}

\begin{remark}
	As in \S\ref{family:degree-7},
	we may view $\phi_{d,13}$ as a deformation
	of an endomorphism of a superelliptic Jacobian.
	We embed~$\ZZ[\alpha_{13}]$ in $\ZZ[\zeta_{13}]$,
	identifying $\alpha_{13}$ with $1 + \zeta_{13}^3 + \zeta_{13}^9$;
	then $\lambda_1 = 1 + \zeta_{13}^7 + \zeta_{13}^8 + \zeta_{13}^{11}$.
	At $t = 0$, both $\family{X}_{d,13}$ and $\family{Y}_{d,13}$
	specialize to the family~$\family{Z}_{d,13}$
	of~\S\ref{family:cyclotomic-CM},
	while $\family{C}_{d,13}$
	specializes to 
	\[
		C_0 
		= 
		\sum_{i\in\{0,7,8,11\}}\!\!\!\!\!\!\variety{y_1-y_2, \zeta_{13}^{i}x_1-x_2}
		\subset
		\product[\QQ(\alpha_{13})(s_2,\ldots,s_d)]{\family{Z}_{d,13}}{\family{Z}_{d,13}} 
		.
	\]
	Each $\variety{y_1-y_2, \zeta_{13}^{i}x_1-x_2}$ 
	induces the automorphism 
	$\zeta^i:(x,y)\mapsto(\zeta_{13}^i x,y)$
	of~$\Jacfamily{\family{Z}_{d,13}}$,
	so 
	\[
		\phi_{C_0} = [1] + \zeta^7 + \zeta^8 + \zeta^{11}
		= 
		\multiplication[\Jacfamily{\family{Z}_{d,13}}]{\lambda_1};
	\]
	hence $\phi_{d,13}$
	is a one-parameter deformation 
	of~$\multiplication[\Jacfamily{\family{Z}_{d,13}}]{\lambda_1}$,
	which splits~$\multiplication[\Jacfamily{\family{Z}_{d,13}}]{3}$.
\end{remark}

\section{
	Genus $\genusfn{d}{15}$ families 
	from Theorem~\ref{theorem:CNC-theorem} (4)
}
\label{family:degree-15}

Consider Theorem~\ref{theorem:CNC-theorem}(4):
Let~$\alpha_{15}$ be an element of $\QQbar$ satisfying
\[
	\alpha_{15}^2 - \alpha_{15} + 4 = 0 ,
\]
so
$\QQ(\alpha_{15}) = \QQ(\sqrt{-15})$;
the involution $\sigma : \alpha_{15} \mapsto 4/\alpha_{15}$ 
generates $\Gal{\QQ(\alpha_{15})/\QQ}$.
Let
\[
	f_{15}(x) 
	=
	x^{15}
	+
	15(\alpha_{15} - 1) t x^{13}
	+
	15(\alpha_{15} + 7) t x^{12}
	+ 
	\cdots
\]
be the polynomial of degree $15$ over $\QQ(\alpha_{15})[t]$
defined in the file~\texttt{degree-15.m}
(so~$f_{15} = 15g$,
where $g$ is the polynomial
of~\cite[\S5.4]{Cassou-Nogues--Couveignes} with $a_1 = \alpha_{15}$
and $T = t$).
We have a factorization
\(
	f_{15}(x_1) - (-f_{15}^\sigma(x_2)) 
	= 
	A_{15}(x_1,x_2)B_{15}(x_1,x_2)
\),
where $A_{15}$ and $B_{15}$
are absolutely irreducible polynomials
of total degree $7$ and $8$ respectively
(also defined in~\texttt{degree-15.m}),
with
$\tau(A_{15}) = A_{15}^\sigma$
and 
$\tau(B_{15}) = B_{15}^\sigma$.

\begin{proposition}
\label{proposition:phi-d-15}
	Let $d > 1$ be an integer,
	and consider the families 
	defined by
	\[
	\begin{array}{c}
		\family{X}_{d,15} : P_d(y_1) = f_{15}(x_1) 
		,\quad 
		\family{Y}_{d,15} : P_d(y_2) = f_{15}^\sigma(x_2) 
		,\medskip \\
		\family{C}_{d,15}
		=
		\variety{y_1-y_2,A_{15}(x_1,x_2)}
		\subset
		\product[\QQ(\alpha_{15})(s_2,\ldots,s_d,t)]{\family{X}_{d,15}}{\family{Y}_{d,15}}
		.
	\end{array}
	\]
	The induced homomorphism
	\(
		\phi_{d,15}
		:=
		\phi_{\family{C}_{d,15}}: 
		\Jacfamily{\family{X}_{d,15}} 
		\to~\Jacfamily{\family{Y}_{d,15}}
	\)
	is a $d$-dimensional family
	of
	\(
		\isogenytypetwo{4}{\genusfn{d}{15}-\genusfn{d}{5}-\genusfn{d}{3}} {2}{2(\genusfn{d}{5}+\genusfn{d}{3})}
	\)-isogenies.
\end{proposition}
\begin{proof}
	Both $\family{X}_{d,15}$
	and $\family{Y}_{d,15}$
	have genus $\genusfn{d}{15}$,
	with $d$ moduli by Lemma~\ref{lemma:family-dimension-with-t}.
	We compute
	$\differentialblock{A_{15}}{14}$
	(given in~\texttt{degree-15.m})
	using Algorithm~\ref{algorithm:differential-matrix-block}.
	We find 
	\[
		\differentialblock{A_{15}}{14}
		\differentialblock{\tau(A_{15})}{14}
		=
		\differentialblock{A_{15}}{14}
		\differentialblock{A_{15}}{14}^\sigma
		=
		4I_{14}
	\]
	(using $\tau(A_{15}) = A_{15}^\sigma$),
	so $\phi_{d,15}$
	splits multiplication-by-$4$
	by Lemma~\ref{lemma:isogeny-classification}.
	After specializing $t$,
	Algorithm~\ref{algorithm:Gak}
	gives 
	$G(A_{15},k) \cong \isogenytypetwo{4}{2(k-m(k))}{2}{4m(k)}$,
	where $m(k) = \#\{i : 1\le i \le k, \gcd(i,15)\not= 1\}$,
	for each $1 \le k \le 14$.
	Each of the $\genusfn{d}{15}$ points $(i,j)$ in $\Newtonint{d,15}$
	therefore contributes a factor of 
	either $\isogenytype{4}{2}$ 
	or $\isogenytype{2}{4}$
	to $(\ker(\phi_{d,15}))^2$,
	according to whether $\gcd(j,15) = 1$ or not.
	The number of points $(i,j)$ in $\Newtonint{d,15}$
	with $\gcd(j,15)\not=1$ is equal to $\genusfn{d}{3} + \genusfn{d}{5}$,
	so
	\[
		(\ker\phi_{d,15})^2
		\cong
		\isogenytypetwo{4}{2(\genusfn{d}{15}-\genusfn{d}{5}-\genusfn{d}{3})}{2}{4(\genusfn{d}{5}+\genusfn{d}{3})} ;
	\]
	the result follows.
	(See Remark~\ref{remark:deformation-15} 
	for more detail on the kernel structure.)
\end{proof}

\begin{remark}
\label{remark:deformation-15}
	As in \S\ref{family:degree-7} and \S\ref{family:degree-13},
	we may view $\phi_{d,15}$ as a deformation
	of an endomorphism of a superelliptic Jacobian.
	Let $S = \{0,1,2,4,5,8,10\}$;
	we embed 
	$\ZZ[\alpha_{15}]$ in $\ZZ[\zeta_{15}]$,
	identifying 
	$\alpha_{15}$ 
	with 
	$\sum_{i\in S} \zeta_{15}^i$.
	At $t = 0$,
	the family $\family{X}_{d,15}$ specializes to
	$\family{Z}_{d,15} : P_d(y_1) = x_1^{15}$,
	which has an automorphism $\zeta: (x_1,y_1)\mapsto(\zeta_{15}x,y)$,
	while $\family{Y}_{d,15}$ specializes
	to $\family{Z}_{d,15}' : P_d(y_2) = -x_2^{15}$,
	which is isomorphic to $\family{Z}_{d,15}$ 
	via~$\iota: (x_2,y_2) \mapsto (-x_2,y_2)$.
	Meanwhile, $A$
	specializes to $A_0 = \prod_{i\in S}(\zeta_{15}^ix_1 + x_2)$,
	so $\family{C}_{d,15}$ 
	specializes to
	\[
		C_0 
		= 
		\sum_{i\in S}
		\variety{y_1-y_2,\zeta_{15}^ix_1 + x_2}
		\subset 
		\product[\QQ(\alpha_{15})(s_2,\ldots,s_d)]{\family{Z}_{d,15}}{\family{Z}_{d,15}'} 
		,
	\]
	and 
	\(
		\phi_{C_0} 
		= 
		\iota\sum_{i\in S}\zeta^i 
		=
		\iota\multiplication[\Jacfamily{\family{Z}_{d,15}}]{\alpha_{15}}
		.
	\)
	Hence
	$\phi_{d,15}$ is a one-parameter deformation
	of an isogeny isomorphic to the endomorphism
	$\multiplication[\Jacfamily{\family{Z}_{d,15}}]{\alpha_{15}}$.

	We gain further insight into the structure of $\ker \phi_{C_0}$,
	and hence~$\ker\phi_{C}$, 
	by decomposing~$\Jacfamily{\family{Z}_{d,15}}$.
	We may view
	$\Jacfamily{\family{Z}_{d,5}}$
	and 
	$\Jacfamily{\family{Z}_{d,3}}$
	as abelian subvarieties of $\Jacfamily{\family{Z}_{d,15}}$
	via the covers $\family{Z}_{d,15} \to \family{Z}_{d,5}$
	and $\family{Z}_{d,15} \to \family{Z}_{d,3}$,
	defined by
	$(x_i,y_i) \mapsto (x_i^3,y_i)$
	and 
	$(x_i,y_i) \mapsto (x_i^5,y_i)$,
	respectively.
	The endomorphism $\psi = {\iota}\circ{\phi_{C_0}}$
	of $\Jacfamily{\family{Z}_{d,15}}$
	is induced by $\variety{y_1-y_2,A_0(x_1,-x_2)}$.
	The matrix $\differentialblock{A_0(x_1,-x_2)}{14}$ is diagonal:
	\[ 
		\differentialblock{A_0(x_1,-x_2)}{14}
		=
		\mathrm{diag}(
		    \alpha_{15}^\sigma,
		    \alpha_{15}^\sigma,
		    2,
		    \alpha_{15}^\sigma,
		    -2,
		    2,
		    \alpha_{15},
		    \alpha_{15}^\sigma,
		    2,
		    -2,
		    \alpha_{15},
		    2,
		    \alpha_{15},
		    \alpha_{15}
		)
		.
	\]
	Considering Eq.~\eqref{eq:differential-matrix-decomposition},
	we see that
	\(
		\differentialmatrix{\psi}(\omega_{i,j})
		=
		2\omega_{i,j}
	\)
	whenever $j = 3$, $6$, $9$, and $12$ (that is,
	when $\omega_{i,j}$
	is the pullback of a differential on $\family{Z}_{d,5}$),
	so
	$\psi$
	acts as $\multiplication[\Jacfamily{\family{Z}_{d,15}}]{2}$
	on~$\Jacfamily{\family{Z}_{d,5}} \subset \Jacfamily{\family{Z}_{d,15}}$.
	Similarly,
	$\differentialmatrix{\phi_{C_0}}(\omega_{i,j}) = -2\omega_{i,j}$
	for $j = 5$ and $10$
	(when $\omega_{i,j}$ is the pullback
	of a differential on $\family{Z}_{d,3}$),
	so
	$\psi$ acts as $\multiplication[\Jacfamily{\family{Z}_{d,15}}]{-2}$
	on $\Jacfamily{\family{Z}_{d,3}} \subset \Jacfamily{\family{Z}_{d,15}}$.
	Looking at the other entries on the diagonal,
	we see that
	$\psi$ acts as multiplication-by-$\alpha_{15}$
	on the $(\genusfn{d}{15}-\genusfn{d}{5}-\genusfn{d}{3})$-dimensional
	complimentary subvariety $\family{A}$ of 
	$\product{\Jacfamily{\family{Z}_{d,3}}}{\Jacfamily{\family{Z}_{d,5}}}$
	in $\Jacfamily{\family{Z}_{d,15}}$.
	This gives us a clearer description of the isomorphism
	in the proof of Proposition~\ref{proposition:phi-d-15}:
	the factors~$\isogenytype{4}{\genusfn{d}{15}-\genusfn{d}{3}-\genusfn{d}{5}}$,
	$\isogenytype{2}{2\genusfn{d}{3}}$,
	and~$\isogenytype{2}{2\genusfn{d}{5}}$
	correspond to $\ker(\psi|_\family{A})$,
	$\ker(\phi|_{\Jacfamily{\family{Z}_{d,3}}})$,
	and~$\ker(\phi|_{{\Jacfamily{\family{Z}_{d,5}}}})$
	respectively.
\end{remark}

\section{
	Genus $\genusfn{d}{21}$ families 
	from Theorem~\ref{theorem:CNC-theorem} (5)
}
\label{family:degree-21}

Consider Theorem~\ref{theorem:CNC-theorem}(5):
Let~$\alpha_{21}$ be an element of $\QQbar$
satisfying
\[
	\alpha_{21}^2 - \alpha_{21} + 2 = 0 ,
\]
so
$\QQ(\alpha_{21}) = \QQ(\sqrt{-7})$;
the involution $\sigma: \alpha_{21} \mapsto 2/\alpha_{21}$ 
generates $\Gal{\QQ(\alpha_{21})/\QQ}$.
Let
\[
	f_{21}(x)
	=
	x^{21}
	+
	(42 \alpha_{21}+ 42) x^{19}
	+
	(84 \alpha_{21}+ 84) x^{18}
	+ 
	(2331 \alpha_{21} - 861) x^{17}
	+ \cdots
\]
be the polynomial of degree~$21$ over $\QQ(\alpha_{21})$
defined in the file~\texttt{degree-21.m}
(such that
$f_{21}(x) = 2^{21}g(x/2)$, 
where $g$ is the polynomial of~\cite[\S5.5]{Cassou-Nogues--Couveignes}
with $a_1 = \alpha_{21}$).
We have a factorization
$f_{21}(x_1) - f_{21}^\sigma(x_2) = A_{21}(x_1,x_2)B_{21}(x_1,x_2)$,
where 
\[
	\begin{array}{r@{\;}l}
	A_{21}(x_1,x_2) = 
	&
	x_1^5 + (\alpha_{21} + 1) x_1^4 x_2 + 2 \alpha_{21} x_1^3 x_2^2 + (10 \alpha_{21} + 18) x_1^3 
	\\ & {}
	+ (2 \alpha_{21} - 2) x_1^2 x_2^3 + (32 \alpha_{21} - 8) x_1^2 x_2 + (20 \alpha_{21} + 4) x_1^2 
	\\ & {}
	+ (\alpha_{21} - 2) x_1 x_2^4 + (32 \alpha_{21} - 24) x_1 x_2^2 + (32 \alpha_{21} - 16) x_1 x_2 
	\\ & {}
	+ (107 \alpha_{21} + 55) x_1 - x_2^5 + (10 \alpha_{21} - 28) x_2^3 + (20 \alpha_{21} - 24) x_2^2 
	\\ & {}
	+ (107 \alpha_{21} - 162) x_2 + 136 \alpha_{21} - 68
	.
	\end{array}
\]
Both $A_{21}$ and $B_{21}$ are absolutely irreducible,
and~$\tau(A_{21}) = -A_{21}^\sigma$
and~$\tau(B_{21}) = B_{21}^\sigma$.

\begin{proposition}
\label{proposition:phi-d-21}
	Let $d > 1$ be an integer,
	and consider the families 
	defined by
	\[
	\begin{array}{c}
		\family{X}_{d,21} : P_d(y_1) = f_{21}(x_1) 
		,\quad  
		\family{Y}_{d,21} : P_d(y_2) = f_{21}^\sigma(x_2) 
		, \medskip \\
		\family{C}_{d,21} = 
		\variety{y_1-y_2,A_{21}(x_1,x_2)}
		\subset 
		\product[\QQ(\alpha_{21})(s_2,\ldots,s_d)]{\family{X}_{d,21}}{\family{Y}_{d,21}}
		.
	\end{array}
	\]
	The induced homomorphism
	\(
		\phi_{d,21}
		:=
		\phi_{\family{C}_{d,21}}: 
		\Jacfamily{\family{X}_{d,21}} 
		\to~\Jacfamily{\family{Y}_{d,21}}
	\)
	is a $(d-1)$-dimensional family
	of
	\( 
		\isogenytypetwo{4}{\genusfn{d}{21}-\genusfn{d}{3}}{2}{2\genusfn{d}{3}} 
	\)-isogenies.
\end{proposition}
\begin{proof}
	Both
	\(\family{X}_{d,21}\) and \( \family{Y}_{d,21} \)
	have genus $\genusfn{d}{21}$,
	with $d-1$ moduli by Lemma~\ref{lemma:family-dimension-no-t}.
	We compute
	$\differentialblock{A_{21}}{20}$
	(given in~\texttt{degree-21.m})
	using Algorithm~\ref{algorithm:differential-matrix-block}.
	We find that
	\[
		\differentialblock{A_{21}}{20}
		\differentialblock{\tau(A_{21})}{20}
		=
		\differentialblock{A_{21}}{20}
		\differentialblock{A_{21}}{20}^\sigma
		=
		4I_{20}
	\]
	(since $\tau(A_{21}) = -A_{21}^\sigma$),
	so $\phi_{21}$ splits multiplication-by-$4$
	by Lemma~\ref{lemma:isogeny-classification}.
	Applying Algorithm~\ref{algorithm:Gak},
	we see that
	\(
		G(A_{21},k) 
		\cong 
		\isogenytypetwo{4}{k-\lfloor{k/7}\rfloor}{2}{2\lfloor{k/7}\rfloor}
	\)
	for 
	$1 \le k \le 20$.
	Hence each point $(i,j)$ in $\Newtonint{d,21}$
	contributes a factor of 
	either $\isogenytype{4}{2}$ or $\isogenytype{2}{4}$ 
	to $(\ker(\phi_{d,21}))^2$,
	according to whether $7$ divides $j$ or not.
	Therefore
	\[
		(\ker\phi_{d,21})^2
		\cong
		\isogenytypetwo{4}{2(\genusfn{d}{21}-\genusfn{d}{3})}{2}{4\genusfn{d}{3}}
		,
	\]
	and the result follows.
\end{proof}

\section{
	Genus $\genusfn{d}{31}$ families from Theorem~\ref{theorem:CNC-theorem} (6)
}
\label{family:degree-31}

Consider Theorem~\ref{theorem:CNC-theorem}(6):
Let~$\alpha_{31}$ and $\beta_{31}$ be elements of $\QQbar$
satisfying
\[
	\beta_{31}^3 - 13\beta_{31}^2 + 46\beta_{31} - 32 = 0 
	\text{\quad and\quad }
	\alpha_{31}^2 - 1/2(\beta_{31}^2 - 7\beta_{31} + 4)\alpha_{31} + \beta_{31} = 0 
	.
\]
Note that 
$\QQ(\alpha_{31})$ is a sextic CM field,
and $\QQ(\beta_{31})$ is its totally real cubic subfield.
The involution $\sigma: \alpha_{31} \mapsto \beta_{31}/\alpha_{31}$
generates $\Gal{\QQ(\alpha_{31}/\QQ(\beta_{31})}$.
Let
\[
	\begin{array}{r@{\;}l}
	f_{31}(x) = 
	&
	x^{31} 
	- 31(\frac{1}{4} (\beta_{31}^2 - 5 \beta_{31} - 10) \alpha_{31} - (\beta_{31}^2 - 7 \beta_{31} + 12)) x^{29} 
	\\ & {}
	- 31(\frac{1}{2} (\beta_{31}^2 - 5 \beta_{31} - 10) \alpha_{31} - (2 \beta_{31}^2 - 14 \beta_{31} + 24)) x^{28} 
	+ \cdots
	\end{array}
\]
be the polynomial of degree~$31$ over~$\QQ(\alpha_{31})$
defined in the file~\texttt{degree-31.m}
(such that $f_{31}(x) = 2^{31}g(x/2)$,
where $g$ is the polynomial of~\cite[\S5.6]{Cassou-Nogues--Couveignes}
with $a_1 = \alpha_{31}$).
We have a factorization
$f_{31}(x_1) - f_{31}^\sigma(x_2) = A_{31}(x_1,x_2)B_{31}(x_1,x_2)$,
where
$A_{31}$ and $B_{31}$ are absolutely irreducible polynomials
of total degree $15$ and $16$, respectively,
with~$\tau(A_{31}) = -A_{31}^\sigma$
and~$\tau(B_{31}) = B_{31}^\sigma$.

\begin{proposition}
\label{proposition:phi-d-31}
	Let $d > 1$ be an integer,
	and consider 
	the families 
	defined by
	\[
	\begin{array}{c}
		\family{X}_{d,31} : P_d(y_1) = f_{31}(x_1) 
		,\quad  
		\family{Y}_{d,31} : P_d(y_2) = f_{31}^\sigma(x_2)
		, \medskip \\
		\family{C}_{d,31} = \variety{y_1-y_2,A_{31}(x_1,x_2)}
		\subset \product[\QQ(\alpha_{31})(s_2,\ldots,s_d)]{\family{X}_{d,31}}{\family{Y}_{d,31}}
		.
	\end{array}
	\]
	The induced homomorphism
	\(
		\phi_{d,31} 
		:= 
		\phi_{\family{C}_{d,31}}
		:
		\Jacfamily{\family{X}_{d,31}} 
		\to~\Jacfamily{\family{Y}_{d,31}}
	\)
	is a $(d-1)$-dimensional family of 
	\(\isogenytypethree{8}{\genusfn{d}{31}/3}{4}{2\genusfn{d}{31}/3}{2}{2\genusfn{d}{31}/3}\)-isogenies.
\end{proposition}
\begin{proof}
	Both
	\(\family{X}_{d,31} \)
	and
	\( \family{Y}_{d,31} \)
	have genus $\genusfn{d}{31}$, with
	$d-1$ moduli by Lemma~\ref{lemma:family-dimension-no-t}.
	We compute
	$\differentialblock{A_{31}}{30}$
	(given in~\texttt{degree-31.m})
	using Algorithm~\ref{algorithm:differential-matrix-block}.
	We see that
	\[
		\differentialblock{A_{31}}{30}
		\differentialblock{\tau(A_{31})}{30}
		=
		\differentialblock{A_{31}}{30}
		\differentialblock{A_{31}}{30}^\sigma
		=
		8I_{30}
	\]
	(using
	$\tau(A_{31}) = -A_{31}^\sigma$),
	so $\phi_{d,31}$ splits multiplication-by-$8$
	by Lemma~\ref{lemma:isogeny-classification}.
	Algorithm~\ref{algorithm:Gak}
	gives
	$G(A_{31},k) \cong (\isogenytypethree{8}{}{4}{2}{2}{2})^{2k}$
	for $1 \le k \le 30$,
	so
	\[
		(\ker(\phi_{d,31}))^6
		\cong
		\left(\isogenytypethree{8}{}{4}{2}{2}{2}\right)^{2\genusfn{d}{31}}
		;
	\]
	the result follows.
\end{proof}

\section{
	Absolute simplicity
}
\label{section:simplicity}

We want to verify
that our isogenies
$\phi: \Jacfamily{\family{X}} \to \Jacfamily{\family{Y}}$
do not arise from products of isogenies of lower-dimensional abelian varieties.
To this end,
where possible,
we show that the generic fibres of~$\Jacfamily{\family{X}}$
and~$\Jacfamily{\family{Y}}$
are absolutely simple.

\begin{proposition}
\label{proposition:simplicity}
	The generic fibres of $\Jacfamily{\family{X}_{d,n}}$
	and $\Jacfamily{\family{Y}_{d,n}}$
	are absolutely simple for
	\begin{enumerate}
	\item	$n = 7$ 
		and all $d \ge 2$; 
	\item	$n = 11$
		and all prime $d \not= 11$;
	\item	$n = 13$
		and all $d \ge 2$ 
	\item	$n = 15$ 
		and all prime $d \notin \{3,5,7\}$;
	\item	$n = 21$
		and all prime $d \notin \{3,5,7\}$;
	\item	$n = 31$
		and all prime $d \notin \{3,5,31\}$.
	\end{enumerate}
\end{proposition}
\begin{proof}
	We need only prove absolute simplicity 
	for each $\Jacfamily{\family{X}_{d,n}}$
	(the existence of the isogeny $\phi_{d,n}$ then
	implies that $\Jacfamily{\family{Y}_{d,n}}$
	is absolutely simple).
	If $\Jacfamily{\family{X}_{d,n}}$ is reducible,
	then so are all of its specializations;
	so it suffices to
	exhibit an absolutely simple specialization 
	of $\Jacfamily{\family{X}_{d,n}}$.
	We can do this for many $(d,n)$
	by applying results of Zarhin 
	to hyperelliptic or superelliptic specializations.
	For~$n = 7$ and~$13$,
	we specialize at~$t = 0$;
	then
	we apply
	\cite[Theorem 1.1]{Zarhin--erJccpl}
	for~$d \ge 5$,
	and 
	\cite[Theorem 1.2]{Zarhin--sJ}
	for~$d = 3$ and~$4$.
	(We cannot use this approach for~$n = 15$,
	because the specialization 
	at~$t = 0$
	is always reducible: cf.~Remark~\ref{remark:deformation-15}.)
	For $n = 11, 15, 21$, and~$31$
	and all prime $d$ not dividing $n(n-1)$
	we specialize at $(s_2,\ldots,s_d) = (0,\ldots,0)$
	and apply \cite[Corollary~1.8]{Zarhin--esJ}.
	For 
	$(d,n) = (2,7)$, $(2,21)$, and~$(2,31)$,
	we specialize at $s_2 = 0$ and apply
	\cite[Theorem 2.3]{Zarhin--hJwCMdtpgpr}.
	For some of the remaining cases,
	we can use the fact that $\family{X}_{d,n}$
	is defined over a number field;
	by~\cite[Lemma~6]{Chai--Oort},
	it suffices to exhibit
	an absolutely simple reduction
	of a specialization of~$\Jacfamily{\family{X}_{d,n}}$
	modulo a prime of good reduction.
	We prove absolute simplicity of reductions
	by computing Weil polynomials
	(using 
	Gaudry and G\"urel's algorithm~\cite{Gaudry--Gurel}
	for superelliptic curves,
	and the Magma system's implementation~\cite{Harrison}
	of Kedlaya's algorithm~\cite{Kedlaya}
	for hyperelliptic curves)
	and applying~\cite[Proposition~3]{Howe--Zhu}.
	For~$(d,n) = (2,11)$
	we specialize at $s_2 = 0$
	and reduce at a prime over~$7$;
	for~$(d,n) = (2,13)$
	we specialize at $(s_2,t) = (1,0)$
	and reduce at a prime over~$53$;
	for~$(d,n) = (2,15)$
	we specialize at $(s_2,t) = (0,1)$
	and reduce at a prime over~$17$;
	and for~$(d,n) = (5,11)$
	we specialize 
	at $(s_2,\ldots,s_5) = (0,\ldots,0)$
	and reduce at a prime over~$31$.
\end{proof}

The list of values of $n$ and $d$ in 
Proposition~\ref{proposition:simplicity}
is not intended to be exhaustive;
it simply reflects the practical and theoretical limits
of the results used in the proof.
We would like to prove simplicity for at least all prime~$d$;
but
the Gaudry--G\"urel algorithm
requires $n$ and $d$ to be coprime,
so we cannot apply it
to cases such as $(d,n) = (11,11)$.
Further, the reduction of a superelliptic Jacobian can only be simple
if the residue field
contains a primitive $d^\mathrm{th}$ root of unity
(otherwise the superelliptic automorphism
does not commute with Frobenius,
so the endomorphism ring is noncommutative,
so the reduction is not simple).
This rules out many small primes of reduction,
rendering the computation much more expensive.
Computing Weil polynomials for 
$(d,n) = (7,15)$, $(5,21)$, $(3,31)$, and $(5,31)$
will therefore require
highly optimised implementations
and significant computing resources.


\end{document}